\documentclass[final,11pt]{elsarticle}

\usepackage{graphics}
\usepackage{tikz}
\usepackage{mathrsfs}
\usepackage{pgfplots}

\usepackage{diagbox,citesort}
\usepackage{makecell}

\usepackage{amssymb}
\usepackage{algpseudocode,algorithm}
\usepackage{amsthm}

\usepackage{amsmath}
\usepackage{mathtools}
\usetikzlibrary{calc}
\newcommand\dbarD{%
\tikz[baseline=(nodeAnchor.base)]{
    \node[inner sep=0] (nodeAnchor) {$D$}; 
    \draw[line width=0.1ex,line cap=round] 
        ($(nodeAnchor.north west)+(0.0em,0.25ex)$) 
            --
        ($(nodeAnchor.north east)+(0.0em,0.25ex)$) 
        ($(nodeAnchor.north west)+(0.0em,0.55ex)$) 
            --
        ($(nodeAnchor.north east)+(0.0em,0.55ex)$) 
    ;
}}

\usepackage{xcolor,enumitem}

\newtheorem{theorem}{Theorem}
\newtheorem{remark}{Remark}

\definecolor{labelkey}{rgb}{0.6,0,1}
\def\edges{{\mathcal E}}

\newcounter{bla}

\journal{Journal of computational Physics}
\usepackage[top=1.75cm,bottom=1.5cm,left=2.5cm,right=2.5cm]{geometry}

\begin{document}

\begin{frontmatter}




\title{A cost-effective nonlinear extremum-preserving finite volume scheme for highly anisotropic diffusion on Cartesian grids, with application to radiation belt dynamics.}



\author[a]{Nour Dahmen\corref{author}}
\author[b]{J\'er\^ome Droniou}
\author[c]{François Rogier}

\cortext[author] {Corresponding author.\\\textit{E-mail address:} nourallah.dahmen@onera.fr}
\address[a]{ONERA/DPHY, Université de Toulouse, Toulouse, France}
\address[b]{School of Mathematics, Monash University, Clayton, Australia}
\address[c]{ONERA/DTIS, Université de Toulouse, Toulouse, France}

\begin{abstract}
We construct a new nonlinear finite volume (FV) scheme for highly anisotropic diffusion equations, that satisfies the discrete minimum-maximum principle. The construction relies on the linearized scheme satisfying less restrictive monotonicity conditions than those of an $M$-matrix, based on a weakly regular matrix splitting and using the Cartesian structure of the mesh (extension to quadrilateral meshes is also possible). The resulting scheme, obtained by expressing fluxes as nonlinear combinations of linear fluxes, has a larger stencil than other nonlinear positivity preserving or minimum-maximum principle preserving schemes. Its larger ``linearized'' stencil, closer to the actual complete stencil (that includes unknowns appearing in the convex combination coefficients), enables a faster convergence of the Picard iterations used to compute the solution of the scheme. Steady state dimensionless numerical tests as well as simulations of the highly anisotropic diffusion in electron radiation belts show a second order of convergence of the new scheme and confirm its computational efficiency compared to usual nonlinear FV schemes.
\end{abstract}


\end{frontmatter}
\section{Introduction}

Finite volume methods have been favored for a long time in engineering communities for solving diffusion equations resulting from fluid flows. Unlike finite difference methods, they are specifically designed to provide a local conservativity property to the numerical resolution, fundamental to conservation laws \cite{eymard2000}. However, important numerical properties related to accuracy, computational efficiency and physical relevance are also required to ensure a reliable numerical solution \cite{droniou2014}.

One of the most sought after properties is monotonicity, which embodies (for linear schemes) the respect of the discrete minimum-maximum principle (MMP) and the preservation of the positivity of the numerical solution \cite{droniou2014}. Monotonicity is obtained by ensuring a monotonic discretization matrix. However, this property is hardly achieved by linear FV schemes for any mesh type or in case of highly anisotropic diffusion problems (with anisotropy ratios $>10^6$), typically encountered in applications like reservoir engineering, radionuclide migration or radiation belt dynamics. For the last application precisely, the numerical constraints are sometimes so high that an incomplete diffusive frame is considered to prevent negative values and spurious non-physical oscillations from occurring \cite{varotsou2008,dahmen2020a,dahmen2020b}.

On quadrilateral grids, linear schemes with 9-point stencils cannot be monotonic in case of highly anisotropic diffusion tensor \cite{KER81,BC05}; recovering monotonicity requires the usage of sometimes unreasonably large stencils \cite{LP09} (which comes with a computational cost). To preserve local stencils and the monotonicity of the scheme, it is necessary to abandon the principle of linearity of the scheme and adopt a flexible nonlinear ``monotonic'' construction based on a convex combination of linear upwind fluxes \cite{droniou2014,schneider2018}, inspired by \cite{BER05}. The convex combination weights are properly chosen to either ensure extremum principle preserving properties, or at least non negative numerical solutions \cite{DLP10,sheng2011,sheng20,LEP10,yuan2008,LEP09,CAN13,LIP09-II,LIP12,LIP07}. This usually amounts to ensuring that a linearized version of the scheme has a so-called $M$-matrix structure, which satisfies the following properties:
\begin{itemize}[itemsep=0pt]
    \item Irreducible with strict positive diagonal entries,
    \item Non negative diagonal entries ($Z$-matrix),
    \item Weak diagonal dominance (strict for at least one row).
\end{itemize}
However, turning to nonlinear schemes impose a trade-off as the positivity preserving principle and the MMP are no longer equivalent. Some issues can also arise when needing to eliminate intermediate unknowns in the construction, several options being available but not all of them leading to preserving the monotonicity of the scheme \cite{sheng20,agelas09,sheng2011,sheng18}. Besides, nonlinear resolution introduces an additional computational cost that can quickly become substantial, especially in the context of radiation belt diffusion, as reported in recent studies \cite{dahmen2020a}.

Starting from the general quest of designing robust numerical schemes for highly anisotropic and inhomogeneous diffusion equations, and the specific considerations imposed by radiation belts dynamics that imperatively requires physically relevant solutions and optimal resolution methods, we design a new nonlinear FV full-monotonicity preserving scheme on non-uniform Cartesian grids (as used in radiation belts simulations \cite{varotsou2008,subbotin2009,su2010}). The main innovation of our design, compared to usual nonlinear FV scheme, is the usage of convex combinations of upwind linear fluxes with a larger  stencil. Precisely, the latter matches in its ``linearized form'' the complete nonlinear stencil accounting for the unknowns appearing in the convex combination coefficients. As a consequence, this feature leads to an important reduction of the number of nonlinear iterations required to compute the scheme's solution. The usage of a larger stencil prevents the linearized schemes to have an $M$-matrix. The monotonicity of our improved method is however preserved using the conditions described in \cite{nordbotten2007}, which are less restrictive than the conditions imposed by an $M$-matrix construction, and are based on a weakly regular splitting of the discretization matrix. Thanks to this construction, the new scheme preserves both positivity and the MMP, and numerically achieves a second order rate of convergence.

The outline of this paper is as follows. In section \ref{sec:current}, we describe the studied diffusion equation, the FV mesh and the construction of typical nonlinear FV schemes found in the literature. The construction of the new scheme, called Relaxed NonLinear MultiPoint Flux Approximation (R-NLMPFA) due to its usage of relaxed monotonicity conditions, is the purpose of Section \ref{sec:scheme}, in which we also recall the alternative monotonicity conditions underpinning it. In Section \ref{sec:tests}, numerical results are presented on steady state dimensionless cases, assessing the numerical properties of the new scheme compared to typical nonlinear schemes, with a special attention paid to the computational cost. In section \ref{sec:radbelt}, we succinctly present our application context of electron radiation belt modelling. The latter relies on a highly inhomogenous, anisotropic diffusion equation inducing huge numerical constraints on finite difference based codes. Then, the R-NLMPFA is tested over a practical case of electron radiation belt simulation to confirm its relevance. Finally, we summarize our work in section \ref{sec:conclusion} and give some conclusions. 

\section{Nonlinear finite volume schemes for anisotropic diffusion equations}\label{sec:current}

\subsection{Diffusion equation and finite volume construction on a Cartesian mesh}
We consider the following general steady state diffusion problem
\begin{equation}
\left\{
    \begin{array}{ll}
        \nabla \cdot (\dbarD\nabla f)+ S = 0 & \mbox{in } \Omega, \\
        f=\overline{f} &\mbox{on } \partial \Omega, \\
    \end{array}
\right.
\label{eq_diffusion}
\end{equation}
with $\Omega \subset \mathbb{R}^2$ a rectangular domain (extension to 3D is possible) and $\partial \Omega $ its boundary on which Dirichlet boundary conditions are imposed. $\dbarD:\Omega\rightarrow M_2(\mathbb{R})$ is the diffusion tensor (uniformly symmetric definite positive and bounded matrix-valued) and $S \in L^2(\Omega)$ is the source term.

The finite volume framework is based on the local balance of fluxes, obtained after integrating the PDE in \eqref{eq_diffusion} over an elementary volume $K\subset \Omega$, and introducing the flux $\overline{F}_{K,\sigma}$ crossing through each edge $ \sigma\in\partial K $, see \cite{eymard2000}:
\begin{equation}
\int_{K} S dV =-\int_{\partial K} \dbarD\nabla f \cdot \boldsymbol n_K dl  = \sum_{\sigma\in\partial K} \overline{F}_{K,\sigma}
\label{eq_balance}
\end{equation}
with
\begin{equation}
\overline{F}_{K,\sigma}=-\int_{\sigma} \dbarD\nabla f \cdot \boldsymbol n_{K,\sigma} dl.
\label{eq_flux}
\end{equation}
Here, $\boldsymbol n_K$ is the outer normal vector to $\partial K$, and $\boldsymbol n_{K,\sigma}$ is the normal vector to the edge $\sigma$ pointing outside $K$.

In FV methods \cite{droniou2014}, $\overline{F}_{K,\sigma}$ is approximated by a consistent expression $F_{K,\sigma}$ on a set $T$ of disjoint mesh cells covering $\Omega$, that is $\overline{\Omega}=\sum_{K\in T}\overline{K}$. Each cell $K\in T$ is is equipped with a point $c_K$ of coordinates $(x_K,y_K)$ (often referred to as its ``center'', although it does not have to be the center of mass) and associated to a set $\edges_K$ of edges such that $\partial K=\sum_{\sigma \in \edges_{K}} \overline{\sigma}$; we denote by $|\sigma|$ the measure of $\sigma$. $\edges_{\rm ext}$ is the set of exterior edges such that $\partial \Omega=\sum_{\sigma \in \edges_{\rm ext}} \overline{\sigma}$, $S_K$ denotes the approximation of $S$ on the cell $K$, $f_{\sigma}$ is the approximation of $f$ on the edge $\sigma$ and $\dbarD(K)$ the diffusion tensor evaluated in $K$ (this evaluation could be at the center of the cell, or the average over the cell).

For $K$ and $L$, two neighbouring cells in the mesh $T$ and sharing the edge $\sigma$, FV methods impose by definition (see \eqref{eq_flux}) the conservativity property
\begin{equation}
F_{K,\sigma}+F_{L,\sigma}=0.
\label{eq_conserv}
\end{equation}
The second key ingredient in the design of a FV method is to impose at the discrete level the flux balance \eqref{eq_balance}, that is:
for all cell $K$,
\[
\sum_{\sigma\in \edges_K}F_{K,\sigma}=S_K.
\]

As this research work is primarily related to the electron radiation belt numerical modelling application, we consider from now on a Cartesian grid case, i.e., a rectangular and conforming but possibly nonuniform mesh as represented in Figure \ref{fig_mesh}. In fact, this grid type is used in the most recent application of the finite volume method in the electron radiation belt context within ONERA's Salammb\^o code  \cite{varotsou2008,dahmen2020a,dahmen2020b,bourdarie2012}; it is also currently used in the finite-difference based codes of the radiation belts community \cite{subbotin2009,su2010,glauert2014}, so designing a FV method specifically adapted to such a grid can ease transition for said community. As a consequence of the choice of a Cartesian grid, we can easily adjust the cell centers so that lines joining them are parallel to the axes.

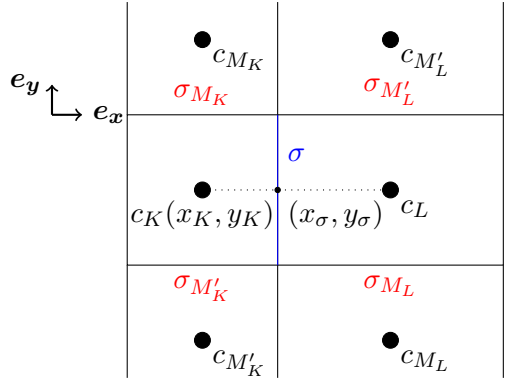
\begin{figure}[ht!]
    \centering
    \begin{tikzpicture}
    \draw (0,0) -- (2,0) -- (2,2) -- (0,2) -- cycle;
    \draw (2,0) -- (5,0) -- (5,2) -- (2,2);
    \draw (0,0) -- (0,-1.5);
    \draw (2,0) -- (2,-1.5);
    \draw (5,0) -- (5,-1.5);
    \draw (0,2) -- (0,3.5);
    \draw (2,2) -- (2,3.5);
    \draw (5,2) -- (5,3.5);
    \draw[blue] (2,0) -- (2,2);
    \draw[dotted] (1,1) -- (3.5,1);
    \filldraw [black] (1,1) circle (3pt) node[anchor=north] {$c_K(x_K,y_K)$};
    \filldraw [black] (3.5,1) circle (3pt) node[anchor=north west] {$c_L$};
    \filldraw [black] (1,3) circle (3pt) node[anchor=north west] {$c_{M_K}$};
    \filldraw [black] (3.5,3) circle (3pt) node[anchor=north west] {$c_{M'_L}$};
    \filldraw [black] (1,-1) circle (3pt) node[anchor=north west] {$c_{M'_K}$};
    \filldraw [black] (3.5,-1) circle (3pt) node[anchor=north west] {$c_{M_L}$};
    \filldraw [blue] (2,1.25) circle (0pt) node[anchor=south west] {$\sigma$};
    \filldraw [black] (2,1) circle (1pt) node[anchor=north west] {$(x_{\sigma},y_{\sigma})$};
    \filldraw [red] (1,2) circle (0pt) node[anchor=south] {$\sigma_{M_K}$};
    \filldraw [red] (1,0) circle (0pt) node[anchor=north] {$\sigma_{M'_K}$};
    \filldraw [red] (3.5,2) circle (0pt) node[anchor=south] {$\sigma_{M'_L}$};
    \filldraw [red] (3.5,0) circle (0pt) node[anchor=north] {$\sigma_{M_L}$};

    \draw[thick,->] (-1,2) -- (-0.6,2) node[anchor= west] {$\boldsymbol{e_x}$};\draw[thick,->] (-1,2) -- (-1,2.4) node[anchor= east] {$\boldsymbol{e_y}$};
    \end{tikzpicture}
    \caption{Cartesian mesh notations}
    \label{fig_mesh}
\end{figure}

\subsection{Design of the numerical fluxes in typical nonlinear FV schemes}
\label{sec:typical}

Nonlinear FV schemes are based on expressing the discrete flux $F_{K,\sigma}$ as a convex combination of two discrete upwind linear fluxes $F_1$ and $F_2$
\begin{equation}
F_{K,\sigma}=\mu_1 F_1+\mu_2 F_2.
\label{eq_pond}
\end{equation}
The weights $\mu_1$ and $\mu_2$ of the convex combination are such as $\mu_1,\mu_2\ge0$ and $\mu_1+\mu_2=1$; their choice is made to either ensure positivity (commonly called monotonic scheme) or the MMP (extremum principle preserving scheme).

$F_1$ and $F_2$ are associated respectively to $K$ and $L$; their expressions are based on equation \eqref{eq_flux} and defined inside their corresponding cell. Since $\dbarD\nabla f\cdot \boldsymbol{n}= \dbarD\boldsymbol{n}\cdot \nabla f$ (as $\dbarD$ is symmetric), this yields the following approximations
\begin{equation}
F_1=-\dbarD(K)\boldsymbol{n}_{K,\sigma}\cdot (\nabla f)_{K}|\sigma|\quad\mbox{ and }\quad
F_2=-\dbarD(L)\boldsymbol{n}_{L,\sigma}\cdot (\nabla f)_{L}|\sigma|.
\label{eq_flux_discret}   
\end{equation}
Recalling that we use a Cartesian grid, consider the case $\boldsymbol{n}_{K,\sigma}=\boldsymbol{e_x}$ and the diffusion tensor expressed as
\begin{equation}
\dbarD=
    \begin{pmatrix}
    D_{xx}&D_{xy}\\
    D_{xy}&D_{yy}\\
    \end{pmatrix}.
\end{equation}
Thus
\begin{equation}
    \begin{aligned}
    \dbarD(K)\boldsymbol{n}_{K,\sigma}=\dbarD(K)\boldsymbol{e_x}=D_{xx}(K)\boldsymbol{e_x}+D_{xy}(K) \boldsymbol{e_y},\\
    \dbarD(L)\boldsymbol{n}_{L,\sigma}=-\dbarD(L)\boldsymbol{e_x}=-D_{xx}(L)\boldsymbol{e_x}-D_{xy}(L)\boldsymbol{e_y}.
\end{aligned}
\label{eq_Dn}
\end{equation}
We aim to define $(\nabla f)_{K}$ (respectively $(\nabla f)_{L}$) according to the direction of the vector $\dbarD\boldsymbol{n}_{K,\sigma}$ (respectively $\dbarD\boldsymbol{n}_{L,\sigma}$) and using only the unknowns from cells neighbouring $K$ and whose closure touches $\sigma$, that is $f_K, f_\sigma,f_{\sigma_{M_K}},f_{\sigma_{M'_K}}$, following the notations in Figure \ref{fig_mesh} (respectively $f_L, f_\sigma,f_{M_L},f_{M'_L}$ for the cell $L$). The approximations of $\nabla f$ in each cell are ``upwind'' (thus leading to what we referred to as ``upwind fluxes'') in the sense that they are adapted to the directions of $\dbarD\boldsymbol{n}$ in each cell (see \eqref{eq_Dn}) and to the sign of $D_{xy}$, with an appropriate choice of transverse edge unknowns. The expressions in \eqref{eq_flux_discret} become:
\begin{equation}
    \begin{gathered}
    F_1=\left\{\begin{array}{ll}
    \displaystyle-\left(D_{xx}(K)\frac{f_{\sigma}-f_{K}}{|x_\sigma-x_K|}+D_{xy}(K)\frac{f_{\sigma_{M_K}}-f_{K}}{|y_{\sigma_{M_K}}-y_K|}\right)|\sigma|&\quad\mbox{ if }D_{xy}(K)\geq 0,\\[1.2em]
    \displaystyle-\left(D_{xx}(K)\frac{f_{\sigma}-f_{K}}{|x_\sigma-x_K|}+(-D_{xy}(K))\frac{f_{\sigma_{M'_K}}-f_{K}}{|y_{\sigma_{M'_K}}-y_{K}|}\right)|\sigma|&\quad\mbox{ if } D_{xy}(K)< 0,
    \end{array}\right.\\
    F_2=\left\{\begin{array}{ll}
    \displaystyle-\left(-D_{xx}(L)\frac{f_{\sigma}-f_{L}}{|x_\sigma-x_L|}-D_{xy}(L)\frac{f_{\sigma_{M_L}}-f_{L}}{|y_{\sigma_{M_L}}-y_L|}\right)|\sigma|&\quad\mbox{ if }D_{xy}(L)\geq 0,\\[1.2em]
    \displaystyle-\left(-D_{xx}(L)\frac{f_{\sigma}-f_{L}}{|x_\sigma-x_K|}-(-D_{xy}(L))\frac{f_{\sigma_{M'_L}}-f_{L}}{|y_{\sigma_{M'_L}}-y_{L}|}\right)|\sigma|&\quad\mbox{ if } D_{xy}(L)< 0,
    \end{array}\right.
    \end{gathered}
    \label{eq_F1_F2_edge}
\end{equation}
where, for any edge $\sigma$, $(x_\sigma,y_\sigma)$ are the coordinates of the intersection of $\sigma$ with the line (orthogonal to $\sigma$) joining the centers of the cells on each side of $\sigma$ (see Figure \ref{fig_mesh}).

The approximations of $\nabla f$ can also be recast with cell center unknowns, which yields the following alternative forms of $F_1$ and $F_2$:
\begin{equation}
    \begin{gathered}
    F_1=\left\{\begin{array}{ll}
    \displaystyle-\left(D_{xx}(K)\frac{f_{L}-f_{K}}{|x_L-x_K|}+D_{xy}(L)\frac{f_{M_K}-f_{K}}{|y_{M_K}-y_K|}\right)|\sigma|&\quad\mbox{ if }D_{xy}(K)\geq 0,\\[1.2em]
    \displaystyle-\left(D_{xx}(K)\frac{f_{L}-f_{K}}{|x_L-x_K|}+(-D_{xy}(K))\frac{f_{M_K'}-f_{K}}{|y_{M_K'}-y_{K}|}\right)|\sigma|&\quad\mbox{ if } D_{xy}(K)< 0,
    \end{array}\right.\\
    F_2=\left\{\begin{array}{ll}
    \displaystyle-\left(-D_{xx}(L)\frac{f_{K}-f_{L}}{|x_L-x_K|}-D_{xy}(L)\frac{f_{M_L}-f_{L}}{|y_{M_L}-y_L|}\right)|\sigma|&\quad\mbox{ if }D_{xy}(L)\geq 0,\\[1.2em]
    \displaystyle-\left(-D_{xx}(L)\frac{f_{K}-f_{L}}{|x_L-x_K|}-(-D_{xy}(L))\frac{f_{M_L'}-f_{L}}{|y_{M_L'}-y_{L}|}\right)|\sigma|&\quad\mbox{ if } D_{xy}(L)< 0.
    \end{array}\right.
    \end{gathered}
    \label{eq_F1_F2_cell_centered}
\end{equation}
A first scheme that can be constructed using the flux-weighing approach is the NonLinear Two-Point Flux approximation scheme (NLTPFA) \cite{droniou2014,schneider2018,yuan2008,lepotier2005}. This scheme achieves with the expressions \eqref{eq_pond} and \eqref{eq_F1_F2_edge} a two-point flux structure
\begin{equation}
    F_{K,\sigma}=\alpha_{K,L}(f)f_{K}-\beta_{K,L}(f)f_{L},
    \label{eq_NLTPFA}
\end{equation}
with $\alpha_{K,L}(f),\,\beta_{K,L}(f)\geq 0$, and $\alpha_{K,L}(f)=\beta_{L,K}(f)$ to ensure \eqref{eq_conserv}. The two-point structure is obtained through a special choice of $\mu_1$ and $\mu_2$ that eliminates the edge unknowns (see \ref{sec:appenA}). It provides a discretization matrix (once the nonlinear coefficients $\alpha_{K,L}$ and $\beta_{K,L}$ are frozen) that has an $M$-matrix structure \cite{droniou2014}, which only ensures the positivity of the solution.

To obtain the full monotonicity, a Local Maximum Principle (LMP) structure is required \cite{DLP10}. In the context of nonlinear FV schemes, this is achieved with the NonLinear Multi-Point Flux Approximation scheme (NLMPFA) \cite{DLP10,droniou2014,schneider2018,sheng2011,cances2013}, whose fluxes are
\begin{equation}
    F_{K,\sigma}=\sum_{N}\tau_{K,N}(f_K-f_N),
    \label{eq_NLMPFA}
\end{equation}
with discretization weights $\tau_{K,N}\ge 0$ ($>0$ if $N$ is the cell on the other side of $\sigma$ from $K$), and the sum usually extended to the cells $N$ that share an edge or vertex with $K$. In order to obtain this cell-centered structure for $F_{K,\sigma}$, we use the cell-centered versions of $F_1$ and $F_2$ in \eqref{eq_F1_F2_cell_centered} and we follow the construction reported in \ref{sec:appenA}.

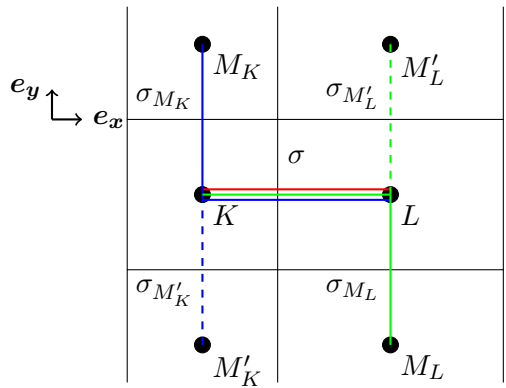
\begin{figure}[ht!]
    \centering
    \begin{tikzpicture}
    \draw (0,0) -- (2,0) -- (2,2) -- (0,2) -- cycle;
    \draw (2,0) -- (5,0) -- (5,2) -- (2,2);
    \draw (0,0) -- (0,-1.5);
    \draw (2,0) -- (2,-1.5);
    \draw (5,0) -- (5,-1.5);
    \draw (0,2) -- (0,3.5);
    \draw (2,2) -- (2,3.5);
    \draw (5,2) -- (5,3.5);
    \filldraw [black] (1,1) circle (3pt) node[anchor=north west] {$K$};
    \filldraw [black] (3.5,1) circle (3pt) node[anchor=north west] {$L$};
    \filldraw [black] (1,3) circle (3pt) node[anchor=north west] {$M_K$};
    \filldraw [black] (3.5,3) circle (3pt) node[anchor=north west] {$M'_L$};
    \filldraw [black] (1,-1) circle (3pt) node[anchor=north west] {$M'_K$};
    \filldraw [black] (3.5,-1) circle (3pt) node[anchor=north west] {$M_L$};
    \filldraw  (2,1.5) circle (0pt) node[anchor=west] {$\sigma$};
    \filldraw  (0.5,2) circle (0pt) node[anchor=south] {$\sigma_{M_K}$};
    \filldraw  (0.5,0) circle (0pt) node[anchor=north] {$\sigma_{M'_K}$};
    \filldraw (3,2) circle (0pt) node[anchor=south] {$\sigma_{M'_L}$};
    \filldraw  (3,0) circle (0pt) node[anchor=north] {$\sigma_{M_L}$};
    \draw[red, thick] (1,1.07)--(3.5,1.07);
    \draw[blue, thick] (1,0.93)--(3.5,0.93);
    \draw[blue, thick] (1,1)--(1,3);
    \draw[blue, thick, dashed] (1,1)--(1,-1);
    \draw[green, thick] (3.5,1)--(3.5,-1);
    \draw[green, thick] (1,1)--(3.5,1);
    \draw[green, thick, dashed] (3.5,1)--(3.5,3);
    \draw[thick,->] (-1,2) -- (-0.6,2) node[anchor= west] {$\boldsymbol{e_x}$};\draw[thick,->] (-1,2) -- (-1,2.4) node[anchor= east] {$\boldsymbol{e_y}$};
    \end{tikzpicture}
    \caption{Construction and stencil unknowns of the $F_{K,\sigma}$ flux discretization on a Cartesian grid, for to the NLTPFA scheme (in red) and for to the NLMPFA scheme (in blue: stencil of $F_1$, in green: stencil of $F_2$ -- solid line refers to the case $D_{xy}\geq 0$, dashed line refers to the case $D_{xy}<0$). Note that, in NLMPFA, $F_2$ is eliminated from the final expression of $F_{K,\sigma}$.}
    \label{fig_NLTPFA_NLMPFA}
\end{figure}
From both schemes, a nonlinear algebric system is obtained that can be written 
\begin{equation}\label{eq:linear_nonlinear}
    A(\boldsymbol{X})\boldsymbol{X}=\boldsymbol{B}
\end{equation}
with $A$ the discretization matrix of nonlinear scheme, $\boldsymbol{X}=(f_K)_{K\in T}$ the unknown vector and $\boldsymbol{B}$ the right-hand side vector containing the source terms $S_K$ and the boundary terms $\overline{f}_\sigma$. The previous equation is solved using the Picard algorithm (Algorithm \ref{algo:picard}), with $\epsilon$ representing the stopping criteria
\begin{algorithm}[!h]
\caption{Picard algorithm for solving \eqref{eq:linear_nonlinear} \label{algo:picard}}
\begin{algorithmic}[1]
\State Initialize $\boldsymbol{X}^0\geq \boldsymbol{0}$ and $s\leftarrow0$
\While{$\frac{||\boldsymbol{X}^{s+1}-\boldsymbol{X}^s||_{\infty}}{||\boldsymbol{X}^s||_{\infty}}\geq \epsilon$}
  \State Construct the linearized system $A(\boldsymbol{X}^s)\boldsymbol{X}=\boldsymbol{B}(\boldsymbol{X}^s)$ (after freezing the nonlinear terms of the scheme)
  \State Solve $A(\boldsymbol{X}^{s})\boldsymbol{X}^{s+1}=\boldsymbol{B}(\boldsymbol{X}^{s})$
    \State $s\leftarrow s+1$
\EndWhile
\end{algorithmic}
\end{algorithm}

\begin{remark}[Stencil of linearized NLTPFA and NLMPFA]\label{rem:stencil}
After freezing the nonlinear coefficients, both schemes have a linearized stencil of 5 points, which thus does not develop on all unknowns involved in the definitions of the linear fluxes $F_1$ and $F_2$ (which would lead to a 9-point stencil). In fact, the NLTPFA scheme considers only $K$ and $L$ as unknowns in each discretization direction. The NLMPFA scheme, on the other hand, considers only the cell center unknowns involved in $F_1$, that is $K$, $L$ and $M_K$ or $M'_K$ (see Figure \ref{fig_NLTPFA_NLMPFA}), while $F_2$ and its cell center unknowns are only used to estimate the convex combination weights during the nonlinear resolution phase and do not enter explicitly in the linearized form of $F_{K,\sigma}$ (see \ref{sec:appenA}).
\end{remark}

\section{Nonlinear monotonic FV scheme with improved stencil}\label{sec:scheme}

\subsection{Monotonicity conditions}

As previously presented, the monotonicity of nonlinear FV schemes is usually based on the $M$-matrix construction argument, imposing a constrained 5-point linearized stencil of discretization on both NLTPFA and NLMPFA schemes. The new scheme that we present here is based on the weaker monotonicity conditions of \cite{nordbotten2007}, derived from the weakly regular splitting of the discretization matrix \cite{berman1994}.

Based on this decomposition, \cite{nordbotten2007} establishes a set of conditions on the discretization matrix terms to ensure its monotonicity. Precisely, let us consider a discretization matrix of a nine-point locally conservative linear scheme, solving \eqref{eq_diffusion} on a two dimensional quadrilateral grid and that can be written as
\begin{equation}
    \sum_{k=1}^9m_k^{i,j}f_k=Q_{i,j}\mbox{ for all $(i,j)$}.
    \label{eq_m}
\end{equation}
Here, $m_{k}^{i,j}$ represents the discretization weight of the 9-point stencil scheme, associated to the neighbouring cell $k$ (as indexed in Figure \ref{fig_m}) of the grid cell $(i,j)$, and $Q_{i,j}$ is constructed from the source term and boundary conditions.

Assume a global ordering of cells starting from the lower left corner of the mesh grid ($i=1$,$j=1$) and running over each grid row from bottom to top. When reported on the discretization matrix $A$, the terms $m_k^{i,j}$ refers to the matrix elements disposed in the row of cell $(i,j)$ (so $m_1^{i,j}$ is the diagonal element, $m_2^{i,j}$ and $m_6^{i,j}$ the first secondary diagonal terms of rank $1$ and $-1$, etc.).

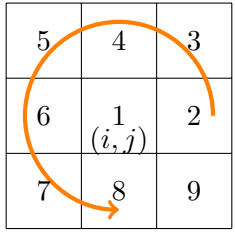
\begin{figure}[ht!]
    \centering
    \begin{tikzpicture}
    \draw[step=1cm,black, thin] (0,0) grid (3,3);
    \draw[ultra thick, ->, orange] (2.75,1.5) arc (0:270:1.25);
    \filldraw [black] (1.5,1.5) circle (0pt) node[anchor=center] {$1$};
    \filldraw [black] (1.5,1.5) circle (0pt) node[anchor=north] {$(i,j)$};
    \filldraw [black] (1.5,2.5) circle (0pt) node[anchor=center] {$4$};
    \filldraw [black] (2.5,1.5) circle (0pt) node[anchor=center] {$2$};
    \filldraw [black] (2.5,2.5) circle (0pt) node[anchor=center] {$3$};
    \filldraw [black] (0.5,1.5) circle (0pt) node[anchor=center] {$6$};
    \filldraw [black] (0.5,0.5) circle (0pt) node[anchor=center] {$7$};
    \filldraw [black] (0.5,2.5) circle (0pt) node[anchor=center] {$5$};
    \filldraw [black] (1.5,0.5) circle (0pt) node[anchor=center] {$8$};
    \filldraw [black] (2.5,0.5) circle (0pt) node[anchor=center] {$9$};
    \end{tikzpicture}
    \caption{Local cell numbering for a nine-point scheme on a Cartesian grid, adapted from \cite{nordbotten2007}.}
    \label{fig_m}
\end{figure}
The monotonicity conditions presented in \cite{nordbotten2007} consists in the following set of 10 inequalities imposed on $m_{k}^{i,j}$:
\begin{gather}
m_1^{i,j}>0, \tag{\textbf{A0}}\\
m_2^{i,j}<0,\tag{\textbf{A1a}}\\
m_{6}^{i,j}<0,\tag{\textbf{A1c}}\\
m_1^{i,j}+m_2^{i,j}+m_6^{i,j}>0,\tag{\textbf{A2}}\\
m_{4}^{i,j}<0,\tag{\textbf{A1b}}\\
m_8^{i,j}<0,\tag{\textbf{A1d}}\\
m_2^{i,j}m_4^{i,j-1}-m_3^{i,j-1}m_1^{i,j}>0,\tag{\textbf{A3a}}\\
m_6^{i,j}m_4^{i,j-1}-m_5^{i,j-1}m_1^{i,j}>0,\tag{\textbf{A3b}}\\
m_2^{i,j}m_8^{i,j+1}-m_9^{i,j+1}m_1^{i,j}>0,\tag{\textbf{A3c}}\\
m_2^{i,j}m_8^{i,j+1}-m_7^{i,j+1}m_1^{i,j}>0.\tag{\textbf{A3d}}
\end{gather}

By interpreting the terms $m_k^{i,j}$ as the elements of $A$, one can notice that:
\begin{itemize}
    \item Condition (\textbf{A0}) imposes a strictly positive main diagonal entry, similar as for an $M$-matrix.
    \item Conditions (\textbf{A1}) impose strictly negative entries only on 4 secondary diagonal entries, unlike the $M$-matrix that requires negative entries on all extra diagonal terms.
    \item Condition (\textbf{A2}) imposes a dominance of the main diagonal entries only on 2 extra-diagonal terms instead of a diagonal dominance by row imposed by the $M$-matrix structure.
    \item Conditions (\textbf{A3}) impose conditions on cross products involving other secondary diagonals.
\end{itemize}

Our idea is to design a nonlinear scheme whose linearized version satisfies these conditions; the scheme is based on the flux weighing method \eqref{eq_pond}, but these weaker conditions enable us to remove the 5-point stencil constraint of NLTPFA and NLMPFA on the linearized matrix.

\subsection{Construction of the new scheme}\label{sec:construction.scheme}

The new nonlinear FV scheme is based on the convex combination of linear fluxes taken in their cell-centered form \eqref{eq_F1_F2_cell_centered}, which can be put in the following general form on a Cartesian grid 
\begin{equation}
\begin{gathered}
    F_1=\lambda_1(f_K-f_L)+\nu_1(f_K-f_M),\\
    F_2=-\lambda_2(f_L-f_K)-\nu_2(f_L-f_N).
    \label{eq_F1_F2_lambda}
\end{gathered}
\end{equation}
Here, $\lambda_1,\lambda_2,\nu_1,\nu_2>0$, $M$ and $N$ refer to the neighbouring cells of $K$ and $L$ chosen accordingly to the signs of $D_{xy}(K)$ and $D_{xy}(L)$ (that is, $M=M_K$ or $M'_K$ and $N=M_L$ or $M'_L$, see Figure \ref{fig_NLTPFA_NLMPFA}). Then, we apply a convex decomposition of the transverse flux branches in $F_1$ and $F_2$
\begin{equation}
\begin{aligned}
F_1={}&\lambda_1(f_K-f_L)
+c_1\nu_{1}(f_K-f_{M})+\underbrace{(1-c_1)\nu_{1}(f_K-f_{M})}_{G_1},\\
F_2={}&-\lambda_{2}(f_L-f_K)
-c_2\nu_{2}(f_{L}-f_N)-\underbrace{(1-c_2)\nu_{2}(f_{L}-f_N)}_{G_2},
\end{aligned}
\end{equation}
with $0<c_1,c_2<1$. The decomposition parameters $c_1$ and $c_2$ are carefully selected (see below) to ensure that the linearized form of the scheme satisfies the monotonicity conditions (\textbf{A0})--(\textbf{A3}). The expression of $F_{K,\sigma}$ becomes
\begin{equation}
    F_{K,\sigma}=(\mu_1\lambda_1+\mu_2\lambda_2)(f_K-f_L)+
    c_1\mu_1\nu_1(f_K-f_{M})+c_2\mu_2\nu_2(f_N-f_L)+
    [\mu_1G_1-\mu_2G_2].
\label{eq_flux_G1G2}
\end{equation}
We construct $\mu_1,\mu_2$ in a usual way, to either annihilate the bracketed expression in the last equation or express it in a more appropriate way:
\begin{equation*}
    \begin{gathered}
    \left\{\begin{array}{ll}
    \displaystyle\mu_1=\frac{|G_2|}{|G_1|+|G_2|}\,,\quad\mu_2=\frac{|G_1|}{|G_1|+|G_2|}&\quad\text{ if }|G_1|+|G_2|\neq 0,\\
    \mu_1=\mu_2=\frac12&\quad\mbox{ otherwise}.
    \end{array}\right.
    \end{gathered}
\end{equation*}
This leads to the following cases:
\begin{itemize}
\item if $G_1G_2\geq 0$ then $G_1|G_2|=|G_1|G_2$ and thus $\mu_1G_1-\mu_2G_2=0$. The expression \eqref{eq_flux_G1G2} becomes
    \begin{equation}
    F_{K,\sigma}=(\mu_1\lambda_1+\mu_2\lambda_2)(f_K-f_L)
    +c_1\mu_1\nu_1(f_K-f_{M})+c_2\mu_2\nu_2(f_N-f_L).
    \label{eq_G1_G2_pos}
    \end{equation}
\item if $G_1G_2 <0$ then $|G_1|G_2=-G_1|G_2|$ so $\mu_1G_1- \mu_2G_2=2\mu_1G_1$. The expression \eqref{eq_flux_G1G2} becomes
\begin{equation}
    F_{K,\sigma}=(\mu_1\nu_1+\mu_2\nu_2)(f_K-f_L)
    +(2-c_1)\mu_1\nu_1(f_K-f_{M})+c_2\mu_2\nu_2(f_N-f_L).
    \label{eq_G1_G2_neg}
\end{equation}
\end{itemize}

These expressions show that the flux discretization retains all the cell centered unknowns used by $F_1$ and $F_2$, namely, $f_K,f_L,f_M$ and $f_N$. This leads, after freezing the non-linearities in $\mu_1$ and $\mu_2$, to a 9-point stencil scheme instead of the 5-point stencil obtained for the NLTPFA and the NLMPFA schemes.

The branches $c_2\mu_2\nu_2(f_N-f_L)$ and $c_2\mu_2\nu_2(f_N-f_L)$ in \eqref{eq_G1_G2_pos}--\eqref{eq_G1_G2_neg} prevent the fluxes from having the form \eqref{eq_NLMPFA}, so the scheme's matrix will not be, in general, an M-matrix. The key idea is however to tune the coefficients $c_1,c_2$ in these branches to ensure that this matrix respect the relaxed monotonicity conditions (\textbf{A0})--(\textbf{A3}). These coefficients can be determined cell-by-cell, or globally using the most demanding local monotonicity conditions. We choose the latter, and we also associate a unique couple $(c_1,c_2)$ to each discretization direction. \ref{sec:appenB} reports the analytical developments followed to obtain such a suitable couple $(c_1,c_2)$, which only depends on the coefficients $\lambda_r,\nu_r$ in the linear fluxes \eqref{eq_F1_F2_lambda} and not on the frozen nonlinear coefficients $\mu_r$ (and thus not on the Picard iterations), see \eqref{eq:conditions.ci}.

Due to the cell-centered discretization logic of the new scheme, similar to the NLMPFA, with ``relaxed'' monotonicity conditions, we refer to the new scheme as \textbf{Relaxed NLMPFA} or \textbf{R-NLMPFA}.

\subsection{Positivity and extremum preserving properties}

We write the scheme in the following form 
\begin{equation}\label{eq:scheme.nl}
    \mathbb{S}(\boldsymbol{U})\boldsymbol{U}=\boldsymbol{R},
\end{equation}
where $\boldsymbol{U}=((f_K)_{K\in T},(\overline{f}_\sigma)_{\sigma\in \edges_{\rm ext}})$ is a vector containing the unknowns of the scheme and the discretized boundary conditions, and $\boldsymbol{R}$ represents the contribution of the volume source $S$. In \eqref{eq:scheme.nl}, the nonlinearities in $\mathbb{S}(\boldsymbol{U})$ correspond to those (frozen in the nonlinear resolution procedure) in $\mu_1,\mu_2$ in \eqref{eq_flux_G1G2}, while the linear term $\boldsymbol{U}$ gathers contributions of $f_K,f_L,f_M,f_N$ explicitly visible in \eqref{eq_G1_G2_pos}--\eqref{eq_G1_G2_neg}.

On the discrete level, the positivity and minimum-maximum preserving principles can be expressed in the following forms
\begin{equation}
\text{if $\boldsymbol{R}\geq 0$ and $\overline{f}_\sigma\geq 0$ for all $\sigma\in \edges_{\rm ext}$, then $f_K\geq 0$ for all $K\in T$},
\label{eq_positivity_principle}
\end{equation}
\begin{equation}
\text{if $\boldsymbol{R}=0$ then $\min_{\sigma\in \edges_{\rm ext}}\overline{f}_\sigma\leq f_K\leq \max_{\sigma\in \edges \rm ext}\overline{f}_\sigma$ for all $K\in T$.}
\label{eq_minimum_maximum_principle}
\end{equation}

\begin{theorem}[Monotonicity of the scheme]
The finite volume scheme based on the fluxes defined by \eqref{eq_F1_F2_lambda}--\eqref{eq_G1_G2_neg}, whose discretization matrix respects (\textbf{A0})--(\textbf{A3}) when its nonlinear terms are frozen, preserves at each nonlinear iteration both positivity \eqref{eq_positivity_principle} and minimum-maximum principle \eqref{eq_minimum_maximum_principle}, achieving consequently full monotonicity.
\end{theorem}
\begin{proof}
The proof of the positivity of the linearized scheme is straightforward due to the monotonicity of its matrix. 
For the extremum principle we consider $\boldsymbol{R}=0$. By the expressions \eqref{eq_G1_G2_pos}--\eqref{eq_G1_G2_neg}, at each linearization step the matrix $\mathbb{S}(\boldsymbol{U}^s)$ satisfies $\mathbb{S}(\boldsymbol{U}^s)\boldsymbol{1}=0$ ($\boldsymbol{1}$ is the vector with all components equal to 1). Thus, the vectors $\boldsymbol{V}=(\max_{\sigma\in \edges_{\rm ext}}\overline{f}_\sigma)\boldsymbol{1}-\boldsymbol{U}^{s+1}$ and $\boldsymbol{W}=\boldsymbol{U}^{s+1}-(\min_{\sigma\in \edges_{\rm ext}}\overline{f}_\sigma)\boldsymbol{1}$ satisfy $v_\sigma\geq 0$ for all $\sigma\in \edges_{\rm ext}$ and $w_\sigma\geq 0$ for all $\sigma\in \edges_{\rm ext}$, and $\mathbb{S}(\boldsymbol{U}^s)\boldsymbol{V}=\mathbb{S}(\boldsymbol{U}^s)\boldsymbol{W}=0$.
Applying the positivity principle to $\boldsymbol{V}$ and $\boldsymbol{W}$ then yields $\min_{\sigma\in \edges_{\rm ext}}\overline{f}_\sigma\leq f_K\leq\max_{\sigma\in \edges_{\rm ext}}\overline{f}_\sigma$.
\end{proof}

\section{Numerical tests for stationary problems}\label{sec:tests}

We test here the numerical performances of R-NLMPFA and we compare them with the results of the NLTPFA and NLMPFA schemes (partially taken from \cite{dahmen2020a,dahmen2020b}). First, we start with a numerical test involving a uniform but highly anisotropic diffusion tensor to measure the computational gain achieved by R-NLMPFA. Then we move to three numerical tests related to the positivity, the minimum principle, and the minimum-maximum principle, using an inhomogeneous highly anisotropic diffusion tensor. Finally, we measure the accuracy of R-NLMPFA with a convergence study.

The following notations are used in the section: $N_u$ is the number of cell center unknowns, $f_{\rm min}$ and $f_{\rm max}$, the extremal values of $f$ inside the domain, and $N_{\rm iter}$ is the number of nonlinear Picard iterations to reach convergence fixed by a stopping criteria of $\epsilon=10^{-6}$. $R_{\rm under}=\frac{N_1}{N_u}$ is the ratio of undershoots, with $N_1$ the number of grid cells where $f_K<0$ when testing positivity, and $N_1$ the number of grid cells where $f_K<\min_{\sigma \in \edges_{\rm ext}}\overline{f}_\sigma$ when testing the minimum principle preserving property. $R_{\rm over}=\frac{N_2}{N_u}$ is the ratio of overshoots, with $N_2$ the number of grid cells where $f_K>\max_{\sigma\in \edges_{\rm ext}}\overline{f}_\sigma$ when testing the maximum principle preserving property. To evaluate the rate of convergence, we use the discrete counterpart of the $L^2$-norm:
\begin{equation}
\|f\|_{2,D}=\left(\sum_{K \in T}|K|f_K^2 \right)^{1/2}
\end{equation}
with $|K|$ the area of $K$. The relative approximation error is calculated based on the previously defined norm: 
\begin{equation}
\mbox{Err}_2=\frac{\|f-f_{\rm ref}\|_{2,D}}{\|f_{\rm ref}\|_{2,D}},
\end{equation}
where $f_{\rm ref}$ is the analytical solution chosen by the method of manufactured solution.

\subsection{Highly anisotropic uniform diffusion tensor test}
First, we consider the problem \eqref{eq_diffusion} in the unit square domain $\Omega=(0,0.5)^2$ with $S=0$. We fix the following diffusion tensor:
\begin{equation}
\dbarD =\begin{pmatrix}
10^7& 10^3 \\
10^3 &1
\end{pmatrix}.
\label{eq_tensor_uniform}
\end{equation}
The latter tensor is uniform, to focus solely on its anisotropy ratio that is $\approx 1,11\cdot 10^7$. We fix $\overline{f}(x,y)=\sin(\pi x)\sin(\pi y)$. The steady state diffusion equation is solved by the NLTPFA, NLMPFA and the R-NLMPFA schemes for three levels of mesh refinements and the Picard algorithm is initialized with unitary function ($f^0=1$). For the R-NLMPFA scheme, we choose the couple $c_1=8.327\cdot 10^{-6}$, $c_2=4.164\cdot 10^{-6}$ that respects the conditions presented in \ref{sec:appenB}. We report in Table \ref{tab_uniform} the number of iterations $N_{\rm iter}$ needed by each scheme to achieve nonlinear convergence.

\begin{table}[h!]
\caption{Anisotropic uniform diffusion tensor test: comparison of NLTPFA, NLMPFA and R-NLMPFA on uniform Cartesian grids.}
\label{tab_uniform}
\centering
\begin{tabular}{|l|c|c|c|c|}
    \hline
    \multicolumn{2}{|c|}{\diagbox[width=10em]{Scheme}{$N_u$}}&$20\times 20$&$40\times 40$&$80\times 80$\\
    \hline
        NLTPFA & $N_{\rm iter}$ &$4$ &$5$ &$8$  \\
	\hline
        NLMPFA & $N_{\rm iter}$ &$14$ &$25$ &$45$  \\
 	\hline
        \textbf{R-NLMPFA}&$N_{\rm iter}$ &$2$ &$2$ &$2$\\
    \hline
    \end{tabular}
\end{table}

The R-NLMPFA scheme requires much fewer nonlinear iterations (up to 6 times fewer) than the NLMPFA, thus preserving the maximum principle at a much lower computational cost. R-NLMPFA shows a similar number of nonlinear iterations as NLTPFA, all the while achieving full monotonicity that NLTPFA lacks.

If, instead of the residual 
\begin{equation}\label{first.residual}
\frac{||\boldsymbol{X}^{s+1}-\boldsymbol{X}^{s}||_\infty}{||\boldsymbol{X}^{s}||_\infty},
\end{equation}
we use in Algorithm \ref{algo:picard} the residual
\begin{equation}\label{second.residual}
\frac{||A(\boldsymbol{X}^{s+1})\boldsymbol{X}^{s+1}-\boldsymbol{B}^{s+1}||_2}{||\boldsymbol{B}^{s}||_2},
\end{equation}
then similar computational performances are obtained (with actually fewer iterations for NLTPFA and NLMPFA); see Table \ref{tab_uniform_1}. The residual \eqref{second.residual} is global and sensitive to the matrix system convergence that can be attained sometimes before the local convergence of the solution measured by \eqref{first.residual}. Given that this latter residual is therefore actually more demanding, we continue using it in the following tests in this section.

\begin{table}[ht!]
\caption{Anisotropic uniform diffusion tensor test using the residual \eqref{second.residual}: comparison of NLTPFA, NLMPFA and R-NLMPFA on uniform Cartesian grids.}
\label{tab_uniform_1}
\centering
\begin{tabular}{|l|c|c|c|c|}
    \hline
    \multicolumn{2}{|c|}{\diagbox[width=10em]{Scheme}{$N_u$}}&$20\times 20$&$40\times 40$&$80\times 80$\\
    \hline
        NLTPFA & $N_{\rm iter}$ &$2$ &$2$ &$3$  \\
	\hline
        NLMPFA & $N_{\rm iter}$ &$4$ &$6$ &$9$  \\
 	\hline
        \textbf{R-NLMPFA} &$N_{\rm iter}$ &$2$ &$2$ &$2$\\
    \hline
    \end{tabular}
\end{table}

\subsection{Positivity preserving test}\label{sec:test.positivity}

We consider the problem \eqref{eq_diffusion} in the unit square domain $\Omega=(0,1)^2$. We fix the following diffusion tensor:
\begin{equation}
\dbarD(x,y) =\frac{1}{x^2+y^2}\begin{pmatrix}
\alpha x^2+y^2 & (\alpha-1)xy \\
(\alpha-1)xy & x^2+\alpha y^2
\end{pmatrix}.
\label{eq_tensor_alpha}
\end{equation}
The latter tensor is nonuniform, has $1$ and $\alpha$ as eigenvalues, and thus an anisotropy ratio of $\frac{1}{\alpha}$; this ratio is fixed at $10^9$ in our tests, as this is the typical magnitude of anisotropy ratios in radiation belts applications \cite{dahmen2020a,dahmen2020b}. The source $S$ adopts the following form:
\begin{equation}
S(x,y)=\left\{
    \begin{array}{ll}
        1 &\text{ if }(x,y)\in [0.25,0.75]\times[0.25,0.75],\\
		0 &\text{ otherwise}.
    \end{array}
\right.
\end{equation}
We fix homogeneous Dirichlet conditions at $x=0$, $y=0$, $y=1$, and a no-flux boundary at $x=1$ (to reflect typical real boundary conditions of radiation belts models). For  the  R-NLMPFA  scheme, we  choose  the  couple $c_1= 2.548\cdot 10^{-5}$, $c_2= 1.274\cdot 10^{-5}$. The results are reported in Table \ref{tab_pos}, in which we can see that all schemes preserve the positivity of the solution; R-NLMPFA however requires fewer Picard iterations $N_{\rm iter}$, with a reduction of up to 27\% compared to NLTPFA and up to 67\% compared to NLMPFA.

\begin{table}[h!]
\caption{Postivity preserving numerical test: comparison of the NLTPFA, NLMPFA and R-NLMPFA on uniform Cartesian grids.}
    \label{tab_pos}
\begin{center}
\begin{tabular}{|l|c|c|c|c|}
    \hline
    \multicolumn{2}{|c|}{\diagbox[width=10em]{Scheme}{$N_u$}}&$20\times 20$&$40\times 40$&$80\times 80$\\
    \hline
    NLTPFA & $f_{\rm min}$ &$6.44\cdot 10^{-10}$ &$7.55\cdot 10^{-15}$ &$2.52\cdot 10^{-23}$\\
    \cline{2-5}
        & $N_{\rm iter}$ &$81$ &$141$ &$222$  \\
	\hline
        NLMPFA & $f_{\rm min}$ &$2.82\cdot 10^{-7}$ &$9.15\cdot 10^{-10}$ &$2.71\cdot 10^{-13}$\\
    \cline{2-5}
        & $N_{\rm iter}$ &$153$ &$310$ &$433$  \\
 	\hline
    \textbf{R-NLMPFA} & $f_{\rm min}$ &$1.57\cdot10^{-7}$ &$3.6\cdot10^{-10}$ &$8.85\cdot10^{-14}$  \\
    \cline{2-5} 
        &$N_{\rm iter}$ &$68$ &$102$ &$193$\\
    \hline
    \end{tabular}
\end{center}
\end{table}

\subsection{Minimum principle preserving test}

The second numerical is related to the minimum preserving principle. We consider the same diffusion problem as in Section \ref{sec:test.positivity}, except for a full domain Dirichlet boundary condition $\overline{f}=1$ on $\partial \Omega$. The results are reported in Table \ref{tab_min}.

\begin{table}[h!]
\caption{Minimum preserving test: comparison of the NLTPFA, NLMPFA and R-NLMPFA on uniform Cartesian grids.}
\label{tab_min}
\centering
\begin{tabular}{|l|c|c|c|c|}
 	\hline
    \multicolumn{2}{|c|}{\diagbox[width=10em]{Scheme}{$N_u$}}&$20\times 20$&$40\times 40$&$80\times 80$\\
    \hline
        NLTPFA & $f_{\rm min}$ &$0.9983$ &$0.9984$ &$0.9987$\\
    \cline{2-5}
        & $N_{\rm iter}$ &$62$ &$83$ &$66$  \\
    \cline{2-5}
        & $R_{\rm under}$&$14.25\%$&$14.31\%$&$13.96\%$\\
    \hline
        NLMPFA & $f_{\rm min}$ &$1.00$ &$1.00$ &$1.00$\\
    \cline{2-5}
        & $N_{\rm iter}$ &$66$ &$107$ &$201$  \\
    \hline
    \textbf{R-NLMPFA} & $f_{\rm min}$ &$1.00$ &$1.00$ &$1.00$\\
    \cline{2-5}
        & $N_{\rm iter}$ &$58$&$93$ &$128$  \\
    \hline
    \end{tabular}
\end{table}

NLMPFA and the R-NLMPFA respect the minimum principle ($f_K\geq 1$ for all $K\in T$) as they satisfy an extremum principle preserving property in addition to the monotonicity of their discretization matrix. This is not the case of the NLTPFA scheme that is only positivity preserving and presents minimum values under 1. The NLTPFA scheme is the most efficient on this test, but at the price of developing many undershoots. Among the two schemes that preserve the minimum principle, the R-NLMPFA outperforms the NLMPFA scheme with up to $36\%$ fewer nonlinear iterations.

\subsection{Minimum-maximum principle preserving test}
For the third test, the domain is reduced to $\Omega=(0,0.5)^2$, and we take $\dbarD$ as in \eqref{eq_tensor_alpha}. We impose $\overline{f}(x,y)=\sin(\pi x)\sin(\pi y)$ on the boundary at $x=0$, $y=0$, $y=0.5$, and no flux condition at $x=0.5$; we also set $S=0$. For the  R-NLMPFA  scheme, the previous couple $c_1= 2.548\cdot 10^{-5}$, $c_2= 1.274\cdot 10^{-5}$ remains valid. The results are synthesised in Table \ref{tab_ext}.

\begin{table}[ht!]
\caption{minimum-maximum preserving test: comparison of the NLTPFA, NLMPFA and R-NLMPFA on uniform Cartesian grids.}
\label{tab_ext}
\centering 
\begin{tabular}{|l|c|c|c|c|}
 	\hline
    \multicolumn{2}{|c|}{\diagbox[width=10em]{Scheme}{$N_u$}}&$20\times 20$&$40\times 40$&$80\times 80$\\
    \hline
    NLTPFA & $f_{\rm min}$ &$2.22\cdot 10^{-15}$ &$ 4.67 \cdot 10^{-27}$ &$3.60\cdot 10^{-44}$\\
    \cline{2-5}
        & $f_{\rm max}$ &$0.9955$ &$0.9988$ &$0.9997$  \\
    \cline{2-5}
        & $N_{\rm iter}$ &$95$ &$181$ &$338$  \\
    \hline
    NLMPFA & $f_{\rm min}$ &$2.10\cdot 10^{-9}$ &$1.63\cdot 10^{-13}$ &$3.73\cdot 10^{-18}$\\
    \cline{2-5}
        & $f_{\rm max}$ &$0.9950$ &$0.9988$ &$0.9995$  \\
    \cline{2-5}
        & $N_{\rm iter}$ &$126$ &$250$ &$606$  \\
    \hline
    \textbf{R-NLMPFA} & $f_{\rm min}$ &$7.30\cdot 10^{-10}$ &$ 4.94 \cdot 10^{-14}$ &$6.80\cdot 10^{-20}$\\
    \cline{2-5}
        & $f_{\rm max}$ &$0.9918$ &$0.9982$ &$0.9995$  \\
    \cline{2-5}
        & $N_{\rm iter}$ &$63$ &$127$ &$271$  \\
    \hline
    \end{tabular}
\end{table}
For this test, all three schemes preserve the minimum and the maximum principles. The R-NLMPFA confirms its efficiency in terms of nonlinear convergence with lower $N_{\rm iter}$ values than NLTPFA (up to 33$\%$ fewer iterations) and NLMPFA (up to 55$\%$ fewer iterations).

\subsection{Convergence study}

We consider the problem \eqref{eq_diffusion} in the domain $\Omega=(0,0.5)^2$, with the diffusion tensor \eqref{eq_tensor_alpha} and $\alpha=10^{-6}$. We fix $f_{\rm ref}(x,y)=\sin(\pi x)\sin(\pi y)$ and $S$ is computed from this exact solution. The results of the convergence study of the three schemes are reported in Figure \ref{fig_convergence} and Table \ref{table_conv_study_1}.

\begin{figure}
\begin{center}
\begin{tikzpicture}
\begin{loglogaxis}[
    title={Convergence study},
    xlabel={$N_u$},
    ylabel={$\mbox{Err}_2(\%)$},
    xmin=09, xmax=110,
    ymin=10^-3, ymax=1,
    xtick={10,100},
    ytick={10^-3,10^-2,10^-1,1},
    legend pos=south west,
    ymajorgrids=true,
    grid style=dashed,
]
\addplot[
    color=blue,
    mark=square,
    ]
    coordinates {
    (20,0.0970)(40,0.0282)(80,0.0073)
    };
\addplot[
    color=green,
    mark=square,
    ]
    coordinates {
    (20,0.3008)(40,0.0767)(80,0.0194)
    };
\addplot[
    color=red,
    mark=square,
    ]
    coordinates {
    (20,0.2658)(40,0.0733)(80,0.0195)
    };
\addplot[
    color=black,
    style=dashed
    ]
    coordinates {
    (20,80/400)(40,80/1600)(80,80/6400)
    };
    \legend{NLTPFA,NLMPFA,R-NLMPFA,2nd order}
\end{loglogaxis}
\end{tikzpicture}
\caption{Rates of convergence of NLTPFA, NLMPFA and R-NLMPFA}
\label{fig_convergence}
\end{center}
\end{figure}
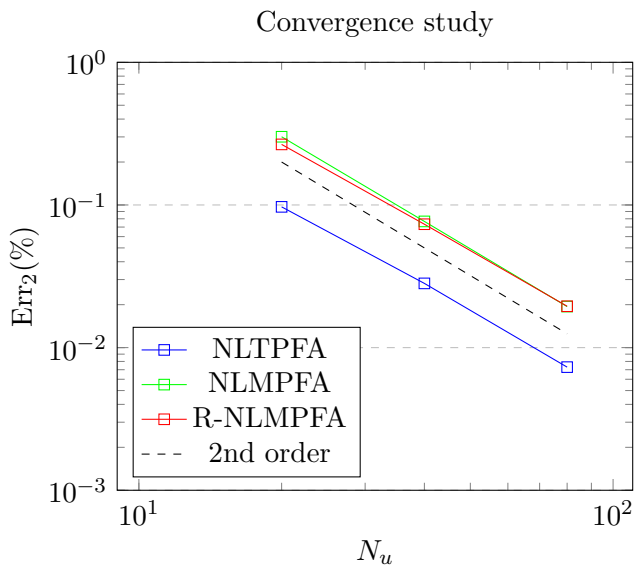

\begin{table}[ht!]
\caption{Convergence study: comparison of the NLTPFA, NLMPFA and R-NLMPFA on uniform Cartesian grids, Picard iterations initialised with $f^0=1$.}
\label{table_conv_study_1}
\centering
\begin{tabular}{|c|c|c|c|}
\hline
\diagbox[width=10em]{Scheme}{$N_u$}&$20\times 20$&$40\times 40$&$80\times 80$\\
\hline
NLTPFA&63&112&185\\
\hline
NLMPFA&79&134&216\\
\hline
\textbf{R-NLMPFA}&$66$ &$101$ &$140$\\\hline
\end{tabular}
\end{table}

The R-NLMPFA ($c_1=2.61\cdot 10^{-5}$, $c_2=1.305\cdot 10^{-5}$) shows a second order of convergence, with error magnitudes similar to the NLMPFA scheme. This might be explained by their similar cell-centered approach. Again the R-NLMPFA is more competitive with up to $35\%$ fewer iterations than the NLMPFA scheme and up to $24\%$ fewer iterations than the NLTPFA scheme.




\section{Application to radiation belts}
\label{sec:radbelt}
Earth's electron radiation belts are structures evolving around Earth and crossing the majority of satellite orbits (see Figure \ref{fig_radbelts}) \cite{horne2013}. Filled with highly energetic electrons trapped by Earth's magnetic field, they can present a rapidly varying and asymmetric dynamic correlated to solar activity.

The usual theoretical development adopted to describe radiation belt dynamics \cite{schulz1974,roederer2016} leads to a 3 dimension diffusion equation, governing the evolution of a phase space density (PSD) function $f$ $(MeV^{-1}.s^{-3})$ and expressed in the space $(y,E,L^*)$ with
\begin{itemize}
    \item  $y=\sin{\alpha_{eq}}$ the sine of the particle’s equatorial pitch angle, the angle between its velocity vector and the magnetic field at the equator (see Figure \ref{fig_radbelts}).
    \item $E$ the particle's energy expressed in $MeV$.
    \item $L^*$ the radial distance of the magnetic field line (dipole approximation) traveled by the trapped particle at the equator (see Figure \ref{fig_radbelts}) \cite{roederer2016}. 
\end{itemize}
\begin{figure}[ht!]
    \centering
    \includegraphics[scale=0.55]{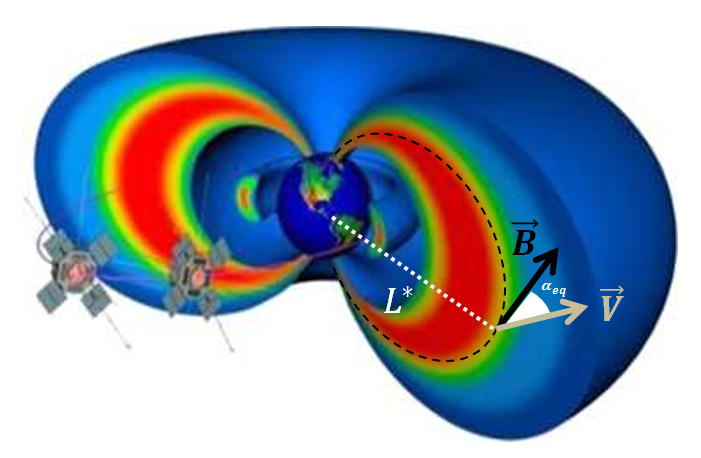}
    \caption{Electron radiation belts representation with the $L^*$ parameter and the equatorial pitch angle $\alpha_{eq}$, in addition to the Van Allen probes orbit, adapted from \cite{kirby2012}.}
    \label{fig_radbelts}
\end{figure}

Due to the multi-scale, localized and highly asymmetric physical processes involved in radiation belt dynamics, the local diffusion frame expressed in $(y,E)$ is spatially inhomogeneous and presents high levels of anisotropy as shown in Figure \ref{fig_tensor} for $\dbarD$ taken at $L^*=4.49$. The 2D tensor anisotropy ratio is also represented in Figure \ref{fig_ratio_aniso}.
\begin{equation}
\frac{\partial f}{\partial t}=\frac{1}{G}\nabla \cdot (G \dbarD\nabla f)
\label{eq_eyl}
\end{equation}
with
\begin{equation}
\dbarD=\begin{pmatrix}
D_{yy}&D_{yE}\\D_{yE}&D_{EE}
\end{pmatrix}\end{equation}
and $G$ the Jacobian of the transformation from the canonical space to $(y,E)$.

\begin{figure}[ht!]
    \centering
    \includegraphics[scale=0.58]{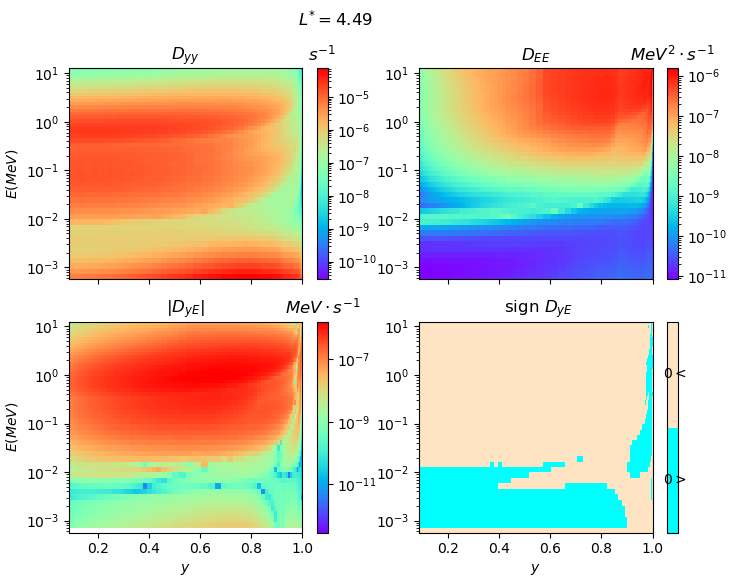}
    \caption{2D mapping of the diffusion tensor terms $D_{yy}$, $D_ {EE}$ and $D_{yE}$, taken at $L^*=4.49$.}
    \label{fig_tensor}
\end{figure}

\begin{figure}[ht!]
    \centering
    \includegraphics[scale=0.75]{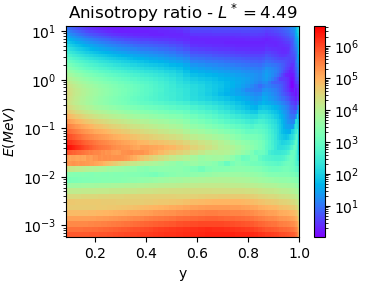}
    \caption{Ratio of anisotropy of the diffusion tensor taken at $L^*=4.49$.}
    \label{fig_ratio_aniso}
\end{figure}
Based on this diffusion model, the radiation belt study community implemented several physical codes  like the pioneering ONERA's Salammb\^o code \cite{varotsou2008,bourdarie2012,beutier1995}, the \textit{British Antarctic Survey} (BAS) Radiation Belt model code \cite{glauert2014}, the \textit{Storm-Time Evolution of Electron Radiation Belt} (STEERB) code \cite{su2010} and the \textit{Versatile Electron Radiation Belt} (VERB) code \cite{subbotin2009}. Cumulative improvements in the modelling of physical processes occurring in the inner magnetosphere led to more and more accurate depictions of the radiation belt dynamics by these codes, especially during the main phases of geomagnetic events. However, their finite difference based numerical resolution started to show its limit to tackle all the induced numerical challenges \cite{dahmen2020a,dahmen2020b}.

Recently, a transition to a finite volume based resolution was initiated in the Salammb\^o code. This was motivated by the suitability of the FV method for diffusion equations and to take advantage of the conservativity property. In order to cope with the high levels of anisotropy and take into account the tricky cross diffusion terms $D_{yE}$ without numerical artefacts, the use of nonlinear finite volume schemes seemed inevitable. Thus, NLPTFA and NLMPFA schemes were adapted and tested over the tricky pitch angle and energy diffusion \eqref{eq_eyl}, with the results reported in \cite{dahmen2020a}.

In this section, we test the R-NLMPFA in the radiation belt context and we focus on its computational cost, which is compared to NLTPFA and NLMPFA.
We consider the transient diffusion problem \eqref{eq_eyl} evaluated over $T_{\rm simu}=90 000s$. We adopt the 2D diffusion tensor $\dbarD$ evaluated at the plan $L^*=4.49$ and showed in Figure \ref{fig_tensor}. Note that the diffusion coefficients were derived from estimations of resonant particle--particle interactions wave--particle (using the WAPI code \cite{sicard2008}).

We start with the initial state $f_{\rm init}=10^{30}$ and we impose the boundary conditions presented in Table \ref{tab:bc}.

The Picard algorithm is initialized with the output of a linear FV scheme, cleaned from overshoots and undershoots, and we adopt the residual \eqref{second.residual}, more adapted to the high dispersion of the solution. We also fix a maximum number of nonlinear iterations $N_{\rm iter, \rm max}=300$, and we consider that the Picard algorithm did not converge above that number. We report in Table \ref{tab_NLTPFA_vs_NLMPFA_vs_NLMONOT_transient} the extremum values of $f$, the average number of Picard iterations per time iteration $N_{\rm iter,\rm avg}$, the percentage of time iterations where the Picard algorithm did not converge ($N_{\rm iter}>300$).

\begin{center}
\begin{table}[ht!]
\caption{Boundary conditions adopted for the transient simulation.}
\label{tab:bc}
\centering
\begin{tabular}{|l|c|c|}
\hline
\emph{Boundary location} & \emph{Boundary condition type}&\emph{Explanation}\\
\hline
$y_{\rm min}=0.0867$& $f=0$& Absence of trapped particles at the atmosphere\\
\hline
$y_{\rm max}=1.0$&$\frac{\partial f}{\partial y}=0$& Equatorial symmetry\\
\hline
$E_{\rm min}=5.664\cdot 10^{-4}~MeV$&$f=10^{30}~ MeV^{-1}s^{-3}$& Low energy seeding\\
\hline
$E_{\rm max}=12.561~MeV$&$f=0$& Limit of magnetic trapping at very high energy\\
\hline
\end{tabular}
\end{table}
\end{center}

\begin{table}[ht!]\renewcommand{\arraystretch}{1.2}
\caption{Transient simulation results with the NLTPFA,NLMPFA and R-NLMPFA schemes for different $\Delta t$, at $T_{\rm simu}=90000s$.}
\label{tab_NLTPFA_vs_NLMPFA_vs_NLMONOT_transient}
\centering
\begin{tabular}{|l|c|c|c|c|}
 	\hline
	\emph{Scheme}& $\Delta t$ &$100s$&$1000s$&$10000s$\\
    \hline
    NLTPFA & $f_{\rm min}$ &$3.96\cdot 10^{26}$ &$4.05\cdot 10^{26}$ &$4.97\cdot 10^{26}$\\
    \cline{2-5}
        & $f_{\rm max}$ &$1.000001191\cdot 10^{30}$ &$1.000000448\cdot 10^{30}$ &$9.99\cdot 10^{29}$  \\
    \cline{2-5}
        & $N_{\rm iter, \rm avg}$ &$1$ &$1$ &$2$  \\
    \hline
    NLMPFA & $f_{\rm min}$ &$3.95\cdot 10^{26}$ &$4.06\cdot 10^{26}$ &$4.96\cdot 10^{26}$\\
    \cline{2-5}
        & $f_{\rm max}$ &$9.99\cdot 10^{26}$ &$9.99\cdot 10^{26}$ &$9.99\cdot 10^{29}$  \\
    \cline{2-5}
        & $N_{\rm iter, \rm avg}$ &$1.23$ &$7.95$ &$242.66$  \\
    \cline{2-5}
        & $\%_{\rm no-conv}$ &$0$ &$0$ &$66.66\%$  \\
    \hline
    R-NLMPFA & $f_{\rm min}$ &$4.12\cdot 10^{26}$&$4.34\cdot 10^{26}$&$5.18\cdot 10^{26}$\\
    \cline{2-5}
        & $f_{\rm max}$ &$9.99\cdot 10^{29}$&$9.99\cdot 10^{29}$&$9.99\cdot 10^{29}$\\
    \cline{2-5}
        & $N_{\rm iter, \rm avg}$ &$1$&$1.01$&$3.11$\\
    \hline
    \end{tabular}
\end{table}

For all the time steps $\Delta t$ we tested, the NLPTFA scheme showed the least $N_{\rm iter,\rm avg}$ values (at maximum 2 Picard iterations per time iteration), followed very closely by the R-NLMPFA (at maximum 4 Picard iterations per time iteration). Hence, the transition from a only positivity preserving scheme to a full monotonicity preserving scheme in our application does not sensibly increase the computational cost. Focusing on the two schemes that also preserve extremum values, the R-NLMPFA prevails with much lower $N_{\rm iter,\rm avg}$ values than the NLMPFA schemes. Most notably, we did not observe with R-NLMPFA any oscillating profiles in the residuals, contrary to those observed with the NLMPFA scheme (see Figure \ref{fig_residue}). This annoying behavior, frequently reported with the NLMPFA scheme, is mainly reponsible for the non-convergence of the Picard algorithm at some time iterations. The R-NLMPFA very likely avoids their occurrence due to its relaxed linearized stencil.

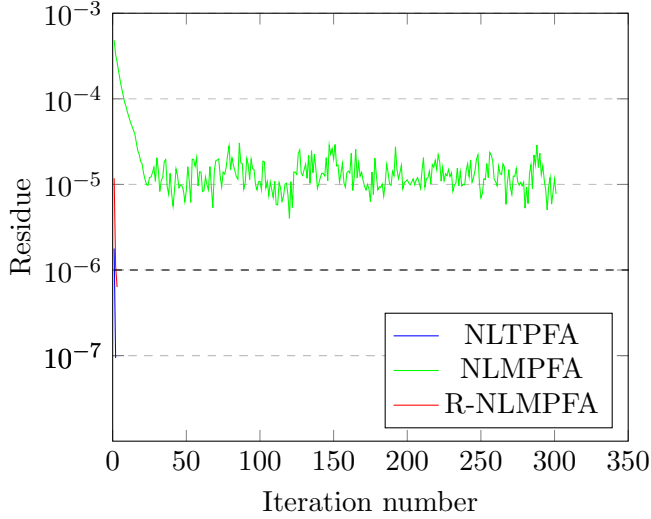
\begin{figure}
\begin{center}
\begin{tikzpicture}
\begin{axis}[
    xlabel={Iteration number},
    ylabel={Residue},
    ymode={log},
    xmin=0, xmax=350,
    ymin=10^-8, ymax=10^-3,
    xtick={0,50,100,150,200,250,300,350},
    ytick={10e-8,10^-7,10^-6,10^-5,10^-4,10^-3,10^-2},
    legend pos=south east,
    ymajorgrids=true,
    grid style=dashed]
\addplot[
    color=blue]
    coordinates {
    (1,1.783686e-6)(2,9.3817305e-8)
    };

\addplot[
    color=green]
    table[x=i,y=n, col sep=comma]{NLMPFA_res.dat};
\addplot[
    color=red]
    coordinates {
    (1,1.18021376e-5)(2,1.125134226e-6)(3,6.31739960e-7)};
    \legend{NLTPFA,NLMPFA,R-NLMPFA}

\addplot[
    color=black,style=dashed]
    coordinates {
    (1,1e-6)(1e3,1e-6)};
    \legend{NLTPFA,NLMPFA,R-NLMPFA}
\end{axis}
\end{tikzpicture}
\end{center}
\caption{Residuals at each Picard iteration step, at $n\times \Delta t=70000s$ for $\Delta t=10000s$}
\label{fig_residue}
\end{figure}

\section{Conclusion}\label{sec:conclusion}
We presented in this paper the construction of a new nonlinear finite volume scheme, R-NLMPFA, adapted to highly anisotropic diffusion equations. This scheme shares the convex combination structure adopted by typical nonlinear FV schemes and a cell-centered formulation. However, it stands out with its monotonicity preserving property obtained by a set of conditions less restrictive than those needed for an M-matrix construction. This construction leads to a full monotonicity preserving scheme, i.e., a positivity and minimum-maximum principle preserving scheme, in addition to a larger linearized stencil over Cartesian grids. Adimensionnal numerical tests involving diffusion tensors with very high anisotropy ratios showed an important reduction of the computational cost with the R-NLMPFA, especially when compared to the analogous extremum preserving scheme NLMPFA. These results were also confirmed on the numerically challenging radiation belts application. The next step consists in generalizing the construction of the R-NLMPFA to three dimensions and/or quadrilateral meshes, which seems achievable since the monotonicity conditions at the core of the construction are still valid for these types of mesh.

\section*{Acknowledgments}
This research was supported by CNES -- The French Space Agency, ONERA -- The French Aerospace Lab, and the University of Monash.

\appendix
\section{Construction of NLTPFA and NLMPFA convex combination weights}\label{sec:appenA}

The NLTPFA scheme construction is based on the forms \eqref{eq_F1_F2_edge} of $F_1$ and $F_2$. Injecting these expressions in \eqref{eq_pond} gives
\begin{equation}
    F_{K,\sigma}=\alpha_{K,L}f_K-\beta_{K,L}f_{L}-(\mu_1a_1(f)-\mu_2a_2(f))
\end{equation}
with $a_1,a_2>0$ containing the edge unknowns, responsible for the loss of the $M$-matrix structure. The expression $\mu_1a_1-\mu_2a_2$ is annihilated by imposing $\mu_1a_1-\mu_2a_2=0$ and $\mu_1+\mu_2=1$, which yields
\begin{equation*}
    \left\{\begin{array}{ll}
    \mu_1=\frac{a_2}{a_1+a_2}\quad\mbox{ and }\quad \mu_2=\frac{a_1}{a_1+a_2}& \quad \text{ if }a_1+a_2\neq 0,\\
    \mu_1=\mu_2=\frac12&\quad\text{ otherwise}.
    \end{array}\right.
\end{equation*}
The NLMPFA scheme construction, on the other hand, is based on the cell-centered form of $F_1$ and $F_2$ as reported in the equation \eqref{eq_F1_F2_cell_centered}, which can be framed in the following general form (with the same notation as in subsection \ref{sec:construction.scheme} -- so $M$ and $N$ are chosen accordingly to the signs of $D_{xy}(K)$ and $D_{xy}(L)$)
\begin{equation*}
    F_1=\lambda_1(f_K-f_L)+\underbrace{\nu_1(f_K-f_M)}_{G_1'}\quad\mbox{ and }\quad
    F_2=-\lambda_2(f_L-f_K)-\underbrace{\nu_2(f_L-f_N)}_{G_2'}.
\end{equation*}
Injecting these expressions in \eqref{eq_pond} yields
\begin{equation}
    F_{K,\sigma}=(\mu_1\lambda_1+\mu_2\lambda_2)(f_K-f_L)+[\mu_1G_1'-\mu_2G_2'].
    \label{eq_NLMPFA_G1_G2}
\end{equation}
The NLMPFA scheme aims to recast the latter flux expression in the multi-point form \eqref{eq_NLMPFA}, by discarding the transverse flux branch $G_2'$ responsible for the loss of the $M$-matrix structure. Thus, $\mu_1,\mu_2$ are chosen to either annihilate the bracketed expression in the last equation, or express it in a form that is appropriate to the LMP structure:
\begin{equation*}
    \begin{gathered}
    \left\{\begin{array}{ll}
    \mu_1=\frac{|G_2'|}{|G_1'|+|G_2'|}\quad\mbox{ and }\quad \mu_2=\frac{|G_1'|}{|G_1'|+|G_2'|} & \quad\mbox{ if }G_1'+G_2'\neq 0,\\[.7em]
    \mu_1=\mu_2=\frac12&\quad\text{ otherwise}.
    \end{array}\right.
    \end{gathered}
\end{equation*}
In other words:
\begin{itemize}
\item if $G_1'G_2'\geq 0$ then $G_1'|G_2'|=|G_1'|G_2'$ and thus
\begin{equation*}
F_{K,\sigma}=(\mu_1\lambda_1+\mu_2\lambda_2)(f_K-f_L).
\end{equation*}
\item if $G_1'G_2' <0$ then $|G_1'|G_2'=-G_1'|G_2'|$ and thus 
\begin{equation*}
F_{K,\sigma}=(\mu_1\lambda_1+\mu_2\lambda_2)(f_K-f_L)+2\mu_1G_1'
=(\mu_1\lambda_1+\mu_2\lambda_2)(f_K-f_L)+2\mu_1\nu_1(f_K-f_{M_K}).
\end{equation*}
\end{itemize}

\section{Choice of the weights $c_1,c_2$ to ensure the monotonicity conditions}\label{sec:appenB}

To present a choice of $c_1$ and $c_2$ in Section \ref{sec:construction.scheme} that ensures Conditions (\textbf{A0})--(\textbf{A3}), we use the notations in Figure \ref{fig_cell_sigma}.
\begin{figure}[ht!]
    \centering
    \begin{tikzpicture}
    \draw (0,0) -- (2,0) -- (2,2) -- (0,2)--cycle;
    \filldraw [black] (1,1) circle (3pt) node[anchor=south] {$c_K(i,j)$};
    \filldraw [blue] (2,1) circle (0pt) node[anchor=west] {$\sigma_1$};
    \filldraw [blue] (1,2) circle (0pt) node[anchor=south] {$\sigma_2$};
    \filldraw [blue] (0,1) circle (0pt) node[anchor=east] {$\sigma_3$};
    \filldraw [blue] (1,0) circle (0pt) node[anchor=north] {$\sigma_4$};
    \draw[thick,->] (-1,2) -- (-0.6,2) node[anchor= west] {$\boldsymbol{e_x}$};\draw[thick,->] (-1,2) -- (-1,2.4) node[anchor= east] {$\boldsymbol{e_y}$};
    \end{tikzpicture}
    \caption{edge notations on a cell $K$.}
    \label{fig_cell_sigma}
\end{figure}
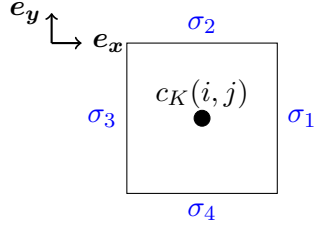
For each $p\in\{1,2,3,4\}$, following \eqref{eq_G1_G2_pos}--\eqref{eq_G1_G2_neg} the flux $F_{K,\sigma_p}$ is expressed, with obvious notations, as

\begin{equation}
F_{K,\sigma_p}=\left\{
\begin{array}{l@{\ }l}
  \begin{array}{l}
  \displaystyle(\mu_1^{\sigma_p}\lambda_1^{\sigma_p}+\mu_2^{\sigma_p}\lambda_2^{\sigma_p})(f_K-f_{L^{\sigma_p}})\\
  \quad\displaystyle+c_1^{\sigma_p}\mu_1^{\sigma_p}\nu_1^{\sigma_p}(f_K-f_{M^{\sigma_p}})
  +c_2^{\sigma_p}\mu_2^{\sigma_p}\nu_2^{\sigma_p}(f_{N^{\sigma_p}}-f_{L^{\sigma_p}})
  \end{array}
  &\mbox{ if } G_1^{\sigma_p}G_2^{\sigma_p}\geq 0\,,\\[1.5em]
    \begin{array}{l}
    \displaystyle(\mu_1^{\sigma_p}\lambda_1^{\sigma_p}+\mu_2^{\sigma_p}\lambda_2^{\sigma_p})(f_K-f_{L^{\sigma_p}})\\
    \quad\displaystyle+(2-c_1^{\sigma_p})\mu_1^{\sigma_p}\nu_1^{\sigma_p}(f_K-f_{M^{\sigma_p}})
    +c_2^{\sigma_p}\mu_2^{\sigma_p}\nu_2^{\sigma_p}(f_{N^{\sigma_p}}-f_{L^{\sigma_p}})
    \end{array}
    &\mbox{ if }G_1^{\sigma_p}G_2^{\sigma_p}<0.
\end{array}\right.
    \label{eq_F_k_sigma_i}
\end{equation}
We unify the previous expression for more clarity by introducing the parameter $\theta^{\sigma_p}$ such that
\[
\theta^{\sigma_p}=\left\{
    \begin{array}{ll}
    c_1^{\sigma_p}&\text{ if }G_1^{\sigma_p}G_2^{\sigma_p}\geq 0\\[.5em]
    2-c_1^{\sigma_p} &\text{ if }G_1^{\sigma_p}G_2^{\sigma_p}<0
    \end{array}\right.
\]
Then \eqref{eq_F_k_sigma_i} becomes
\begin{equation}
F_{K,\sigma_{i}}=
    (\mu_1^{\sigma_p}\lambda_1^{\sigma_p}+\mu_2^{\sigma_p}\lambda_2^{\sigma_p})(f_K-f_{L^{\sigma_p}})
    +\theta^{\sigma_p}\mu_1^{\sigma_p}\nu_1^{\sigma_p}(f_K-f_{M^{\sigma_p}})+c_2^{\sigma_p}\mu_2^{\sigma_p}\nu_2^{\sigma_p}(f_{N^{\sigma_p}}-f_{L^{\sigma_p}}).
    \label{eq_F_k_sigma_i_theta}
\end{equation}
We define a couple $(c_1,c_2)$ on each discretization direction, i.e., a unique couple for each $\sigma_p$ is applied on each mesh cell. Thus, we have to determine four pairs of weighing coefficients. Imposing the conservativity property \eqref{eq_conserv} links the coefficients on both sides of each edge:
\begin{equation}
    c_1^{\sigma_p}=c_2^{\sigma_p+2}\quad\mbox{ and }\quad c_2^{\sigma_p}=c_1^{\sigma_p+2}\quad
\text{ for }p\in {1,2}.
    \label{eq_c1_c2}
\end{equation}
This reduces the number of weights to be determined to 4, namely $(c_1^{\sigma_1},c_2^{\sigma_1},c_1^{\sigma_2},c_2^{\sigma_2})$.

We then apply the condition (\textbf{A0}), that is $m_{1}^{i,j}>0$.
From \eqref{eq_F_k_sigma_i_theta} we find
\begin{equation}\label{formula:m1}
\begin{aligned}
m_1^{i,j}={}&
    \mu_1^{\sigma_1}\lambda_1^{\sigma_1}+\mu_2^{\sigma_1}\lambda_2^{\sigma_1}+\theta^{\sigma_1}\mu_1^{\sigma_1}\nu_1^{\sigma_1}+\mu_1^{\sigma_2}\lambda_1^{\sigma_2}+\mu_2^{\sigma_2}\lambda_2^{\sigma_2}+\theta^{\sigma_2}\mu_1^{\sigma_2}\nu_1^{\sigma_2}\\
    &+\mu_1^{\sigma_3}\lambda_1^{\sigma_3}+\mu_2^{\sigma_3}\lambda_2^{\sigma_3}+\theta^{\sigma_3}\mu_1^{\sigma_3}\nu_1^{\sigma_3}+\mu_1^{\sigma_4}\lambda_1^{\sigma_4}+\mu_2^{\sigma_4}\lambda_2^{\sigma_4}+\theta^{\sigma_4}\mu_1^{\sigma_4}\nu_1^{\sigma_4}.
\end{aligned}
\end{equation}
which can be written in the following convinient form
\[m_1^{i,j}=\sum_{r=1}^4 \left(\mu_1^{\sigma_r}\lambda_1^{\sigma_r}+\mu_2^{\sigma_r}\lambda_2^{\sigma_r}+\theta^{\sigma_r}\mu_1^{\sigma_r}\nu_1^{\sigma_r}\right).\]
This shows, since all these coefficients are strictly positive by construction, that \textbf{(A0)} is always valid.

We now move to conditions (\textbf{A1}):
\[
    m_2^{i,j}<0\,,\quad m_4^{i,j}<0\,,\quad m_6^{i,j}<0\,,\quad m_8^{i,j}<0.
\]
The expressions of these coefficients depend on the location of $M^{\sigma_p}$ in the local cell numbering around the cell $K$ (see Figure \ref{fig_m}), which in turn depends on the sign of $D_{xy}(K)$. To take this into account and later on for \textbf{(A3)}) we define the following parameters
\begin{equation}\label{eq:def.psi}
\psi_{1}^{\sigma_p}=
\begin{cases}
1 \text{ if }D_{xy}(K)\geq 0\\
0 \text{ if }D_{xy}(K)< 0
\end{cases}
\end{equation}
and from \eqref{eq_F_k_sigma_i_theta} we deduce the expressions of $m_2^{i,j},m_4^{i,j},m_6^{i,j},m_8^{i,j}$:
\begin{equation}\label{eq:formulas.meven}
    \begin{aligned}
    m_2^{i,j}={}&-\left[\mu_1^{\sigma_1}\lambda_1^{\sigma_1}+\mu_2^{\sigma_1}\lambda_2^{\sigma_1}+c_2^{\sigma_1}\mu_2^{\sigma_1}\nu_2^{\sigma_1}+\psi_1^{\sigma_2}\theta^{\sigma_2}\mu_1^{\sigma_2}\nu_1^{\sigma_2}+(1-\psi_1^{\sigma_4})\theta^{\sigma_4}\mu_1^{\sigma_4}\nu_1^{\sigma_4}\right],\\
    m_4^{i,j}={}&
    -\left[\mu_1^{\sigma_2}\lambda_1^{\sigma_2}+\mu_2^{\sigma_2}\lambda_2^{\sigma_2}+c_2^{\sigma_2}\mu_2^{\sigma_2}\nu_2^{\sigma_2}+\psi_1^{\sigma_1}\theta^{\sigma_1}\mu_1^{\sigma_1}\nu_1^{\sigma_1}+(1-\psi_1^{\sigma_3})\theta^{\sigma_3}\mu_1^{\sigma_3}\nu_1^{\sigma_3}\right],\\
    m_6^{i,j}={}&-\left[\mu_1^{\sigma_3}\lambda_1^{\sigma_3}+\mu_2^{\sigma_3}\lambda_2^{\sigma_3}+c_2^{\sigma_3}\mu_2^{\sigma_3}\nu_2^{\sigma_3}+\psi_1^{\sigma_4}\theta^{\sigma_4}\mu_1^{\sigma_4}\nu_1^{\sigma_4}+(1-\psi_1^{\sigma_2})\theta^{\sigma_2}\mu_1^{\sigma_2}\nu_1^{\sigma_2}\right],\\
    m_8^{i,j}={}&-\left[\mu_1^{\sigma_4}\lambda_1^{\sigma_4}+\mu_2^{\sigma_4}\lambda_2^{\sigma_4}+c_2^{\sigma_4}\mu_2^{\sigma_4}\nu_2^{\sigma_4}+\psi_1^{\sigma_3}\theta^{\sigma_3}\mu_1^{\sigma_3}\nu_1^{\sigma_3}+(1-\psi_1^{\sigma_1})\theta^{\sigma_1}\mu_1^{\sigma_1}\nu_1^{\sigma_1}\right].
    \end{aligned}
\end{equation}
Similarly as for (\textbf{A0}), the conditions (\textbf{A1}) are then valid for any choice of $c_1$ and $c_2$, since $0<\mu_1,\mu_2,c_1,c_2<1$ and $\nu_1,\nu_2>0$.

Let us now consider (\textbf{A2}), that is, $m_{1}^{i,j}+m_2^{i,j}+m_6^{i,j}>0$
which translates using \eqref{eq_F_k_sigma_i_theta} into 
\[
\begin{gathered}
    \underbrace{(\mu_1^{\sigma_2}\lambda_1^{\sigma_2}+\mu_2^{\sigma_2}\lambda_2^{\sigma_2})+(\mu_1^{\sigma_4}\lambda_1^{\sigma_4}+\mu_2^{\sigma_4}\lambda_2^{\sigma_4})+\theta^{\sigma_1}\mu_1^{\sigma_1}\nu_1^{\sigma_1}+\theta^{\sigma_3}\mu_1^{\sigma_3}\nu_1^{\sigma_3}}_{A}\\-\underbrace{(c_2^{\sigma_1}\mu_2^{\sigma_1}\nu_2^{\sigma_1}+c_2^{\sigma_3}\mu_2^{\sigma_3}\nu_2^{\sigma_3})}_{B}>0.
\end{gathered}
\]
Since $(\mu_1,\mu_2)$ are the coefficients of a convex combination, and $\theta^{\sigma_p}\geq 0$, we derive the following inequalities over the quantities $A$ and $B$
\[
\begin{gathered}
    A\geq \min_{i,j}(\lambda_1^{\sigma_2},\lambda_2^{\sigma_2})+\min_{i,j}(\lambda_1^{\sigma_4},\lambda_2^{\sigma_4}),\\
    B\leq c_2^{\sigma_1}\nu_2^{\sigma_1}+c_2^{\sigma_3}\nu_2^{\sigma_3}\leq(c_2^{\sigma_1}+c_2^{\sigma_3})\max_{i,j}(\nu_2^{\sigma_1},\nu_2^{\sigma_3}),
\end{gathered}
\]
and we obtain a first condition on $(c_2^{\sigma_1},c_2^{\sigma_3})$, that is on $(c_2^{\sigma_1},c_1^{\sigma_1})$ using \eqref{eq_c1_c2}:
\begin{equation}\label{eq:cond.m1}
    c_2^{\sigma_3}+c_2^{\sigma_1}=c_1^{\sigma_1}+c_2^{\sigma_1}<\frac{\min_{i,j}(\lambda_1^{\sigma_2},\lambda_2^{\sigma_2)}+\min_{i,j}(\lambda_1^{\sigma_4},\lambda_2^{\sigma_4})}{{\max_{i,j}}(\nu_2^{\sigma_1},\nu_2^{\sigma_3})}.
\end{equation}
Let us consider (\textbf{A3}):
\[
\begin{gathered}
m_{2}^{i, j} m_{4}^{i, j-1}-m_{3}^{i, j-1} m_{1}^{i, j}>0\,,\quad
m_{6}^{i, j} m_{4}^{i, j-1}-m_{5}^{i, j-1} m_{1}^{i, j}>0\,,\\
m_{6}^{i, j} m_{8}^{i, j+1}-m_{7}^{i, j+1} m_{1}^{i, j}>0\,,\quad
m_{2}^{i, j} m_{8}^{i, j+1}-m_{9}^{i, j+1} m_{1}^{i, j}>0.
\end{gathered}
\]
The expressions of $m_3,m_5,m_7,m_9$ depend on the location of $M^{\sigma_p}$ and $N^{\sigma_p}$ which is defined by the sign of $D_{xy}$ on $K$ and $L^{\sigma_p}$. To take that into account, we define the following parameter:
\begin{equation}
\psi_{2}^{\sigma_p}=
\begin{cases}
1 \text{ if }D_{xy}(L^{\sigma_p})\geq 0,\\
0 \text{ if }D_{xy}(L^{\sigma_p})<0,
\end{cases}
\end{equation}
and from \eqref{eq_F_k_sigma_i_theta} we deduce the expressions of $m_3^{i,j},m_5^{i,j},m_7^{i,j},m_9^{i,j}$:
\[
\begin{aligned}
m_3^{i,j}={}&(1-\psi_2^{\sigma_1}) c_2^{\sigma_1}\mu_2^{\sigma_1}\nu_{2}^{\sigma_1}+(1-\psi_2^{\sigma_2})c_2^{\sigma_2}\mu_2^{\sigma_2}\nu_{2}^{\sigma_2},\\
m_5^{i,j}={}&\psi_2^{\sigma_2}c_2^{\sigma_2}\mu_2^{\sigma_2}\nu_{2}^{\sigma_2}+\psi_2^{\sigma_3}c_2^{\sigma_3}\mu_2^{\sigma_3}\nu_{2}^{\sigma_3},\\
m_7^{i,j}={}&(1-\psi_2^{\sigma_3})c_2^{\sigma_3}\mu_2^{\sigma_3}\nu_{2}^{\sigma_3}+(1-\psi_2^{\sigma_4})c_2^{\sigma_4}\mu_2^{\sigma_4}\nu_{2}^{\sigma_4},\\
m_9^{i,j}={}&\psi_2^{\sigma_1} c_2^{\sigma_1}\mu_2^{\sigma_1}\nu_{2}^{\sigma_1}+c_2^{\sigma_4}\mu_2^{\sigma_4}\nu_{2}^{\sigma_4}.
\end{aligned}
\]
Since $(\mu_1,\mu_2)$ are the coefficients of a convex combination, and $\psi_2^{\sigma_p},\theta^{\sigma_p}\geq 0$, these relations and \eqref{formula:m1}, \eqref{eq:formulas.meven} yield the following inequalities:
\begin{equation}
\begin{gathered}
m_{2}^{i, j} m_{4}^{i, j-1}\geq \min_{i,j}(\lambda_1^{\sigma_1},\lambda_2^{\sigma_1})\min_{i, j-1}(\lambda_1^{\sigma_2},\lambda_2^{\sigma_2}),\\
m_{6}^{i, j} m_{4}^{i, j-1}\geq \min_{i,j}(\lambda_1^{\sigma_3},\lambda_2^{\sigma_3})\min_{i, j-1}(\lambda_1^{\sigma_2},\lambda_2^{\sigma_2}),\\
m_{6}^{i, j} m_{8}^{i, j+1}\geq \min_{i,j}(\lambda_1^{\sigma_3},\lambda_2^{\sigma_3})\min_{i, j+1}(\lambda_1^{\sigma_4},\lambda_2^{\sigma_4}),\\
m_{2}^{i, j} m_{8}^{i, j+1}\geq \min_{i,j}(\lambda_1^{\sigma_1},\lambda_2^{\sigma_1})\min_{i, j+1}(\lambda_1^{\sigma_4},\lambda_2^{\sigma_4}),\\
m_1^{i,j}\leq\underbrace{\max_{i,j}(\lambda_1^{\sigma_1},\lambda_2^{\sigma_1})+\max_{i,j}(\lambda_1^{\sigma_2},\lambda_2^{\sigma_2})+\max_{i,j}(\lambda_1^{\sigma_3},\lambda_2^{\sigma_3})+\max_{i,j}(\lambda_1^{\sigma_4},\lambda_2^{\sigma_4})+2(\nu_1^{\sigma_1}+\nu_1^{\sigma_2}+\nu_1^{\sigma_3}+\nu_1^{\sigma_4})}_{A'},\\
m_3^{i,j-1}\leq (c_2^{\sigma_1}+c_2^{\sigma_2})\max_{i,j-1}(\nu_{2}^{\sigma_1},\nu_{2}^{\sigma_2}),\\
m_5^{i,j-1}\leq(c_{2}^{\sigma_2}+c_{2}^{\sigma_3})\max_{i,j-1}(\nu_{2}^{\sigma_2},\nu_{2}^{\sigma_3}),\\
m_7^{i,j+1}\leq (c_2^{\sigma_3}+c_2^{\sigma_4})\max_{i,j+1}(\nu_{2}^{\sigma_3},\nu_{2}^{\sigma_4}),\\
m_9^{i,j+1}\leq(c_{2}^{\sigma_1}+c_{2}^{\sigma_4})\max_{i,j+1}(\nu_{2}^{\sigma_1},\nu_{2}^{\sigma_4}).
\end{gathered}
\end{equation}
This leads to conditions on $c_2^{\sigma_1}$, $c_2^{\sigma_2}$, $c_2^{\sigma_3}$ and $c_2^{\sigma_4}$:
\[
\begin{gathered}
    c_2^{\sigma_1}+c_2^{\sigma_2}<\frac{\min_{i,j}(\lambda_1^{\sigma_1},\lambda_2^{\sigma_1})\min_{i, j-1}(\lambda_1^{\sigma_2},\lambda_2^{\sigma_2})}{\max_{i,j-1}(\nu_{2}^{\sigma_1},\nu_{2}^{\sigma_2})A'},\\
    c_2^{\sigma_2}+c_2^{\sigma_3}=c_2^{\sigma_2}+c_1^{\sigma_1}<\frac{\min_{i,j}(\lambda_1^{\sigma_3},\lambda_2^{\sigma_3})\min_{i, j-1}(\lambda_1^{\sigma_2},\lambda_2^{\sigma_2})}{\max_{i,j-1}(\nu_{2}^{\sigma_2},\nu_{2}^{\sigma_3})A'},\\
    c_2^{\sigma_3}+c_2^{\sigma_4}=c_1^{\sigma_1}+c_1^{\sigma_2}<\frac{\min_{i,j}(\lambda_1^{\sigma_3},\lambda_2^{\sigma_3})\min_{i, j+1}(\lambda_1^{\sigma_4},\lambda_2^{\sigma_4})}{\max_{i,j+1}(\nu_{2}^{\sigma_3},\nu_{2}^{\sigma_4})A'},\\
    c_2^{\sigma_1}+c_2^{\sigma_4}=c_2^{\sigma_1}+c_1^{\sigma_2}<\frac{\min_{i,j}(\lambda_1^{\sigma_1},\lambda_2^{\sigma_1})\min_{i, j+1}(\lambda_1^{\sigma_4},\lambda_2^{\sigma_4})}{\max_{i,j+1}(\nu_{2}^{\sigma_1},\nu_{2}^{\sigma_4})A'}.
\end{gathered}
\]
Combined with \eqref{eq:cond.m1}, this leads to the five following inequalities to be ensured by $c_1^{\sigma_1},c_2^{\sigma_1},c_1^{\sigma_2},c_2^{\sigma_2}$
\begin{equation}\label{eq:conditions.ci}
\begin{gathered}
c_1^{\sigma_1}+c_2^{\sigma_1}<\frac{\min_{i,j}(\lambda_1^{\sigma_2},\lambda_2^{\sigma_2})+\min_{i,j}(\lambda_1^{\sigma_4},\lambda_2^{\sigma_4})}{\max_{i,j}(\nu_2^{\sigma_1},\nu_2^{\sigma_3})},\\
c_2^{\sigma_1}+c_2^{\sigma_2}<\frac{\min_{i,j}(\lambda_1^{\sigma_1},\lambda_2^{\sigma_1})\min_{i, j-1}(\lambda_1^{\sigma_2},\lambda_2^{\sigma_2})}{\max_{i,j-1}(\nu_{2}^{\sigma_1},\nu_{2}^{\sigma_2})A'},\\
c_2^{\sigma_2}+c_1^{\sigma_1}<\frac{\min_{i,j}(\lambda_1^{\sigma_3},\lambda_2^{\sigma_3})\min_{i, j-1}(\lambda_1^{\sigma_2},\lambda_2^{\sigma_2})}{\max_{i,j-1}(\nu_{2}^{\sigma_2},\nu_{2}^{\sigma_3})A'},\\
c_1^{\sigma_1}+c_1^{\sigma_2}<\frac{\min_{i,j}(\lambda_1^{\sigma_3},\lambda_2^{\sigma_3})\min_{i, j+1}(\lambda_1^{\sigma_4},\lambda_2^{\sigma_4})}{\max_{i,j+1}(\nu_{2}^{\sigma_3},\nu_{2}^{\sigma_4})A'},\\
c_2^{\sigma_1}+c_1^{\sigma_2}<\frac{\min_{i,j}(\lambda_1^{\sigma_1},\lambda_2^{\sigma_1})\min_{i, j+1}(\lambda_1^{\sigma_4},\lambda_2^{\sigma_4})}{\max_{i,j+1}(\nu_{2}^{\sigma_1},\nu_{2}^{\sigma_4})A'}.
\end{gathered}
\end{equation}

\bibliographystyle{elsarticle-num}
\bibliography{biblio}

\end{document}